\documentclass[a4paper]{amsart}
\usepackage[T1]{fontenc}
\usepackage[utf8]{inputenc}
\usepackage{lmodern}
\usepackage{enumerate}
\usepackage{amssymb,amsxtra}
\usepackage[all]{xy}
\usepackage{xcolor}
\usepackage{nicefrac,mathtools}
\usepackage{microtype}
\usepackage{amscd}
\usepackage{geometry}
\usepackage{mathbbol}
\usepackage{comment}

\usepackage{comment}

\usepackage[pdftitle={...},
 pdfauthor={Dami\'an Ferraro},
 pdfsubject={Mathematics}]{hyperref}
\usepackage{cite}

\numberwithin{equation}{section}
\theoremstyle{plain}
\newtheorem{theorem}[equation]{Theorem}
\newtheorem{lemma}[equation]{Lemma}
\newtheorem{proposition}[equation]{Proposition}
\newtheorem{corollary}[equation]{Corollary}
\theoremstyle{definition}
\newtheorem{definition}[equation]{Definition}

\theoremstyle{remark}
\newtheorem{remark}[equation]{Remark}
\newtheorem{example}[equation]{Example}

\newtheorem{alphthm}{Theorem}			

\newcommand*{\congto}{\xrightarrow\sim}

\newcommand*{\C}{\mathbb C}

\newcommand{\car}{\curvearrowright}

\newcommand{\A}{\mathcal{A}}
\newcommand{\Ad}{\operatorname{Ad}}

\newcommand{\bB}{\mathbb{B}}
\newcommand{\bC}{\mathbb{C}}

\newcommand{\bk}{\mathbb{k}}

\newcommand{\bM}{\mathbb{M}}

\newcommand{\cB}{\mathcal{B}}

\newcommand{\cF}{\mathcal{F}}

\newcommand{\cM}{\mathcal{M}}
\newcommand{\cN}{\mathcal{N}}

\newcommand{\dual}[1]{\widehat{#1}}

\newcommand{\id}{\operatorname{id}}

\newcommand{\onto}{\twoheadrightarrow}

\newcommand{\rmu}{{r^{-1}}}
\newcommand{\rt}{\mathtt{rt}}
\newcommand{\smu}{{s^{-1}}}

\newcommand{\tmu}{{t^{-1}}}

\newcommand{\wstar}{{\operatorname{\rm w}^*}}
\newcommand{\wot}{\texttt{wot}}

\newcommand{\M}{\mathcal M} 
\newcommand{\sbe}{\subseteq}
\newcommand{\K}{\mathbb{K}} 

\newcommand*{\into}{\hookrightarrow}

\DeclareMathOperator{\Aut}{Aut}

\providecommand{\cspn}{\overline{\mathop{\rm span}}}

\providecommand{\red}{{\mathop{\rm r}}}
\providecommand{\spn}{{\mathop{\rm span}}}
\providecommand{\supp}{{\mathop{\rm supp}}}

\newcommand*{\Cst}{\mathrm C^*}
\newcommand*{\Wst}{\mathrm W^*}
\newcommand*{\cstar }{\texorpdfstring{\(\Cst\)\nobreakdash-\hspace{0pt}}{*-}}
\newcommand*{\Wstar}{\texorpdfstring{\(\Wst\)\nobreakdash-\hspace{0pt}}{*-}}

\DeclarePairedDelimiterX{\braket}[2]{\langle}{\rangle}{#1\,\delimsize\vert\,\mathopen{}#2}

\date{\today}

\title[W*-Amenability for Fell bundles over discrete groups]{W*-Amenability for Fell bundles over discrete groups}
\author{Alcides Buss}
\author{Damián Ferraro}

\date{\today}

\begin{document}

\begin{abstract}
We investigate amenability for \Wstar Fell bundles over a discrete group $G$, with a focus on its characterization via approximation properties and conditional expectations. Building on the notion of W$^*$-amenability, we construct an enlarged  \Wstar Fell bundle analogous to $\ell^\infty(G, M)$ for a group action $G$ on a von Neumann algebra $M$, and relate amenability to the existence of suitable conditional expectations at both the bundle and crossed-product levels. Our results unify and extend several approaches to amenability for noncommutative dynamical systems.

As applications of our methods, we prove that amenability of Fell bundles passes to restrictions to subgroups and that a Fell bundle over a group $G$ is amenable if and only if both its restriction to a normal subgroup $H \trianglelefteq G$ and the associated quotient Fell bundle over $G/H$ are amenable. This provides a powerful structural tool that extends classical permanence results for group amenability to the setting of \Wstar Fell bundles and also \cstar  algebraic Fell bundles. We also discuss how Fell bundles and their amenability interact with group coactions on $C^*$-algebras and von Neumann algebras.
\end{abstract}

\maketitle


\section{Introduction}

The study of amenability in the context of operator algebras has led to a rich theory with deep 
connections to approximation properties, conditional expectations, and dynamical systems. 
For Fell bundles over discrete groups, the central notion is \emph{Exel's approximation property} (AP) 
\cite{Exel1997Amenability}. This property implies the weak containment property, i.e.\ the equality 
of the full and reduced cross-sectional C$^*$-algebras of the bundle, and it is strongly related to 
nuclearity: if the unit fiber is nuclear, then the nuclearity of a cross-sectional C$^*$-algebra is 
equivalent to the AP of the underlying Fell bundle. 
It remains a major open problem whether the AP is in general equivalent to weak containment.

A fundamental step towards understanding this property is the following observation, proved in \cite{abadie2021amenability}: a Fell bundle 
$\cB = \{B_t\}_{t \in G}$ has the AP if and only if its bidual W$^*$-Fell bundle 
$\cM = \{M_t\}_{t \in G}$, with $M_t = B_t''$, has the W$^*$-approximation property (W$^*$AP). 
Thus, the von Neumann bundle $\cM$ is not merely a technical enlargement of $\cB$, but rather the 
natural framework in which approximation can be studied. This provides the main motivation of the 
present work: we develop a systematic theory of amenability for W$^*$-Fell bundles, with the goal 
of understanding and extending the approximation property for C$^*$-Fell bundles.

For actions of discrete groups on von Neumann algebras, amenability was classically defined by 
Anantharaman-Delaroche \cite{ADaction1979} as the existence of a $G$-equivariant conditional 
expectation
\[
P \colon \ell^\infty(G,M) \to M,
\]
where the lifted action $\tilde{\gamma}$ of $G$ on $\ell^\infty(G,M)$ is given by 
$\tilde{\gamma}_r(f)(s) = \gamma_r(f(r^{-1}s))$. This notion admits several equivalent 
characterizations, including central and quasi-central approximation properties 
\cite{ADsystemes1987,BssEff_amenability}. 
The W$^*$AP for Fell bundles, introduced in \cite{abadie2021amenability}, extends this picture 
from group actions to the more general framework of W$^*$-Fell bundles.

Our first main contribution is a new characterization of amenability for W$^*$-Fell bundles in terms 
of conditional expectations. We construct a canonical enlargement of a given W$^*$-Fell bundle --
analogous to $\ell^\infty(G,M)$ in the case of group actions -- and show that the existence of a 
$G$-equivariant conditional expectation from this enlarged bundle to the original one is equivalent 
to W$^*$AP. This links approximation-theoretic and categorical views of amenability and connects to 
the theory of dual coactions.

\begin{alphthm}\label{thm: new characterization}
    For a W*-Fell bundle $\cM=\{M_t\}_{t\in G},$ the following statements are equivalent.
    \begin{enumerate}[(a)]
        \item\label{item: M WAD amenable} $\cM$ is W$^*$-amenable.
        \item\label{item: cond exp ell infty} There exists a conditional expectation $P\colon \ell^\infty(G,\cM)\to \cM.$
        \item\label{item: cond exp k} There exists a conditional expectation $Q\colon \bk_\wstar(\cM)\to W^*_\red(\cM).$
    \end{enumerate}
    Here $\ell^\infty(G,\cM):=\{\ell^\infty(G, M_t)\}_{t \in G}$ denotes the canonical enlargement of a $W^*$-Fell bundle, to be introduced in Section~\ref{sec:enlarged-Fell-bundles},   $W^*_\red(\cM)$ denotes the W*-crossed section algebra and $\bk_\wstar(\cM)$ the W*-algebra of kernels of $\cM$.
\end{alphthm}

In the second part of the paper, we establish structural permanence results. 
We show that amenability is preserved under restriction to subgroups and can be detected via 
extensions by normal subgroups. These results generalize classical permanence theorems for amenable 
groups to the setting of Fell bundles, and apply equally in the C$^*$ and W$^*$ settings. In this direction, our main result can be summarized as follows:

\begin{alphthm}
Let $\cM = \{M_t\}_{t \in G}$ be a W$^*$-Fell bundle or a C$^*$-algebraic Fell bundle over a discrete group $G$, 
and let $H \leq G$ be a subgroup. If $\cM$ is amenable (i.e. C$^*$-amenable in the von Neumann sense or has the approximation property in the C$^*$-case), then the restricted bundle $\cM|_H$ is also amenable.

Moreover, if $H \trianglelefteq G$ is a normal subgroup, then $\cM$ is amenable if and only if both its 
restriction $\cM|_H$ and the associated quotient Fell bundle over $G/H$ are amenable.
\end{alphthm}

Finally, to illustrate the scope of our methods, we revisit a concrete class of examples: 
Green twisted actions. We show that amenability of such actions is equivalent to W$^*$-amenability of 
the induced action on the center of $B''$. This connects our general theory with a classical 
construction in the theory of crossed products, and demonstrates how the W$^*$-bundle approach 
naturally captures and extends known results.

\medskip

\subsection*{Outline of the paper}
In Section~\ref{sec: WFell bundles} we recall the basic theory of C$^*$- and W$^*$-Fell bundles and give 
concrete descriptions of the associated algebras of kernels and of the reduced cross-sectional 
W$^*$-algebra. In Section~\ref{sec:enlarged-Fell-bundles}, for a given W$^*$-Fell bundle $\cM$ over $G$, 
we construct the enlarged bundle $\ell^\infty(G,\cM)$ and prove that its reduced cross-sectional 
W$^*$-algebra coincides with the W$^*$-algebra of kernels $\bk_\wstar(\cM)$. 
In the last part of Section~\ref{sec: WFell bundles} we use a normal subgroup $H \trianglelefteq G$ 
and construct a W$^*$-Fell bundle $\dot\cM$ over $\dot G = G/H$ such that $\bk_\wstar(\dot\cM)$ is 
isomorphic to the algebra of fixed points of $\bk_\wstar(\cM)$ under the restriction to $H$ of the 
canonical action of $G$. Moreover, our isomorphism maps $W^*_\red(\dot\cM)$ onto $W^*_\red(\cM)$. 
Section~\ref{sec: amenability and fixed points} develops the conditional expectation characterization 
of W$^*$AP and proves the permanence results. Section~\ref{sec:coactions} discusses the role of 
coactions and relates our constructions to duality theory. Finally, in Section~\ref{sec:applications}, 
we adapt our results to the C$^*$-setting and conclude with the case study of Green twisted actions.

\subsection*{Notations and conventions}
For a $C^*$-algebra $B$, we write $\M(B)$ for its multiplier $C^*$-algebra, and $U\M(B)$ for the group of its unitaries. The center of a $C^*$-algebra $B$ will be denoted by $Z(B)$.

All our inner products are linear in the second variable, even those of Hilbert spaces.

For a Hilbert space or a Hilbert $B$-module $X$, we shall write $\bB(X)$ for the $C^*$-algebra of all adjointable $B$-linear operators on $X$. Notice that $\M(B)\cong \bB(B)$ if $B$ is viewed as a Hilbert $B$-module.

\section{W*-Fell bundles over discrete groups}\label{sec: WFell bundles}

We adopt from \cite{abadie2021amenability} the definition of W*-Fell bundles over discrete groups and all the related notions and constructions (W$^*$-amenability, W*AP, W*-crossed product, W*-algebra of kernels, etc).

\begin{remark}
In \cite{abadie2021amenability} these notions were denoted by adding the prefix 
``AD'' (for ``Anantharaman--Delaroche'') -- for instance, 
\emph{AD-amenability} or \emph{W*AD-amenability}, etc. -- to emphasize their connection with the approximation properties introduced by Anantharaman--Delaroche. In the present paper, we simplify the notation by omitting this prefix, writing simply ``\cstar amenable'' and ``\Wstar amenable'', etc., while keeping the same underlying definitions.
\end{remark}

Fix a  W*-Fell bundle $\cM = \{M_t\}_{t \in G}$ and a unital representation $\pi\colon \cM\to \bB(X)$ on a Hilbert space $X$ such that $M_t\to \bB(X),$ $m\mapsto \pi(m)$, is isometric and $\wstar-$continuous for all $t\in G$. The Krein-Smulian Theorem \cite[Corollary 12.6]{Conway_CourseFA} implies that  $\pi(M_t)$ is $\wstar$-closed. Moreover, $M_t \to \pi(M_t)$, $m \mapsto \pi(m)$, is a $\wstar$-homeomorphism.
This is the case because given Banach spaces $U$ and $V$ and a linear map $T\colon U^*\to V^*$, the restriction of $T$ to the closed unit ball $U^*_1$ is $\wstar$-continuous if and only if $T=S^*$ for a linear and bounded $S\colon V\to U$. In particular, $T$ is $\wstar$-continuous if and only if $T|_{U^*_1}$ is $\wstar$-continuous.

The left and right regular representations of $G$ will be denoted $\lambda\colon G\to \bB(\ell^2(G))$ and $\rho\colon G\to \bB(\ell^2(G)),$ respectively.
If we need to specify the group, we write $\lambda^G$ and $\rho^G$.
For any Hilbert space $X$, we identify $X\otimes \ell^2(G)$ with $\ell^2(G,X)$ in the canonical way.

\subsection{A concrete description of the W*-algebra of kernels}

A \emph{kernel} of $\cM$ is a function $k\colon G \times G \to \cM$ such that $k(r, s) \in M_{r\smu}$ for all $r, s \in G$.
The kernels of finite support form a normed $*$-algebra $\bk_c(\cM)$ with the operations
\begin{align*}
    h * k(r, s) & := \sum_{t \in G} h(r, t) k(t, s), &
    h^*(r, s) & := h(s, r)^*, &
    \|h\|_2 & := \Big( \sum_{r, s \in G} \|h(r, s)\|^2\Big)^{1/2}.
\end{align*}

The set $\cF$ of finite subsets of $G$ is directed with the usual inclusion order: $F_1\leq F_2\Leftrightarrow F_1\sbe F_2$.
For every $F = \{t_1, \ldots, t_n\}\in \cF$, the set $\bk_F(\cM)$ of kernels with support contained in $F \times F$ is $*$-isomorphic to the $C^*$-algebra $\bM_{\textbf{t}}(\cM)$ of \cite[Lemma 2.8]{AbFrrEquivalence}, where $\textbf{t} = (t_1^{-1}, \ldots, t_n^{-1})$. 
Thus, $\bk_F(\cM)$ is a $C^*$-algebra and it happens that its $C^*$-norm is equivalent to $\|\ \|_2$.
Notice that $\bk_c(\cM)$ is the union of the $C^*$-algebras $\{ \bk_F(\cM) \colon F\in \cF\}$ and that the $C^*$-algebra of kernels $\bk(\cM)$ is the inductive limit of $\{\bk_F(\cM)\}_{F\in \cF}$, or the completion of $\bk_c(\cM)$ with respect to its unique $C^*$-norm.

As before, we fix a unital, isometric $\wstar$-continuous representation $\pi\colon \M\to \bB(X)$. Each $T \in \bB(X\otimes \ell^2(G))$ has a matrix representation $(T_{r,s})_{r, s \in G}$ with respect to the canonical basis $(\delta_s)_{s\in G}\sbe\ell^2(G)$; the entries $T_{r,s}\in \bB(X)$ are determined by the identity $\langle T_{r, s} x, y \rangle = \langle T(x\otimes \delta_s), y\otimes \delta_r \rangle$. The concrete W*-algebra of kernels (associated to $\pi$) is
\begin{equation*}
    \bk_\wstar^\pi(\cM):=\left\{T \in \bB(X\otimes \ell^2(G))\colon T_{r, s} \in \pi(M_{r \smu})\ \forall \  r, s \in G\right\}.
\end{equation*}

For every $F \in \cF$, we define $\bk_{\wstar F}^\pi(\cM)$ as the set of operators $T \in \bk_\wstar^\pi(\cM)$ such that $T_{r, s} = 0$ if $(r, s) \notin F \times F$. Then $\bk_{\wstar F}^\pi(\cM)$ is a von Neumann algebra that is $C^*$-isomorphic to $\bk_F(\cM).$ The norm-closure of $\cup\{\bk_{\wstar F}^\pi(\cM)\colon F\in \cF\}$, denoted $\bk^\pi(\cM)$,  is $C^*$-isomorphic to the $C^*$-algebra of kernels $\bk(\cM)$ and $\bk_\wstar^\pi(\cM)$ is the von Neumann algebra generated by $\cup\{\bk_{\wstar F}^\pi(\cM)\colon F\in \cF\}$, that is, 
$$\bk_\wstar^\pi(\cM)=\overline{\bk^\pi(\cM)}^{\wot},$$
where $\wot$ denotes is the weak operator topology.
The natural action $\beta$ of $G$ on $\bk(\cM)$ is determined by $$\beta_t(k)(r,s)=k(rt,st),\quad k\in \bk_c(\cM)$$ 
and is implemented as $\text{Ad}(1\otimes \rho)$.
Thus, both $\bk_\wstar^\pi(\cM)$ and $\bk^\pi(\cM)$ are $\text{Ad}(1\otimes \rho)$-invariant.

The following remark will be used many times without an explicit mention, specially in the case where $X$ is a Hilbert space, $Y=X\otimes \ell^2(G)$ and $Y_0=\spn\{x\otimes\delta_r\colon r\in G,\ x\in X\}=C_c(G,X).$

\begin{remark}
Given a Hilbert space $Y$, a dense subspace $Y_0\sbe Y$ and a sesquilinear function $\sigma\colon Y_0\times Y_0\to \bC$;  there exist an operator $T\in \bB(Y)$ such that $\sigma(x,y)=\langle Tx,y\rangle$ for $x,y\in Y_0$ if and only if there exists a constant $C>0$ such that $|\sigma(x,y)|\leq C\|x\|\|y\|$ for all $x,y\in Y_0$. 
\end{remark}

Recall from \cite{abadie2021amenability} that the right $M_e$-Hilbert module $\ell^2_\wstar(\cM)$ is formed by those cross sections $f \colon G \to \cM$ such that $\{\sum_{t \in F} f(t)^* f(t)\}_{F\in \cF}$ is bounded.
The inner product of $f,g\in \ell^2_\wstar(\cM)$ is the $\wstar$-limit $\langle f,g\rangle\equiv \sum_{t\in G}f(t)^*g(t)$ of $\{\sum_{t\in F}f(t)^*g(t)\}_{F\in \cF}$ and the action of $m\in M_e$ on $f$ gives the element $fm\in \ell^2_\wstar(\cM)$ with $fm(t)=f(t)m$.

The module $\ell^2_\wstar(\cM)$ can be faithfully represented as a \wot-closed ternary ring of operators (see \cite{Zl83}) via $\pi^2\colon \ell^2_\wstar(\cM)\to \bB(X,X\otimes\ell^2(G))$ with $(\pi^2(f)x)(t) := \pi(f(t))x$.
Hence, $\bB(\ell^2_\wstar(\cM))$ is a \Wstar algebra and a bounded net $\{T_i\}_{i\in I}\subset \bB(\ell^2_\wstar(\cM))$ $\wstar$-converges to $T\in \bB(\ell^2_\wstar(\cM))$ if and only if for all $f,g\in\ell^2_\wstar(\cM)$, $\wstar\lim_i \langle T_if,g\rangle  = \langle Tf,g\rangle$; which is equivalent to say that $\wot\lim_i \pi(\langle T_if,g\rangle )=\pi(\langle Tf,g\rangle )$.

Notice that $\ell^2(\cM)$, the closure in $\ell^2_\wstar(\cM)$ of the finitely supported cross sections $C_c(\cM)$, is a $\wstar$-dense submodule of $\ell^2_\wstar(\cM)$ and recall that the canonical *-homomorphism $\Phi\colon \bk(\cM)\to \bB(\ell^2(\cM))$ \cite{Ab03} is given by $\Phi(k)f(r)=\sum_{s\in G}k(r,s)f(r)$ ($k$ and $f$ with finite support).
The representation $\pi^2$ induces $\Phi^\pi\colon \bk(\cM)\to \bB(X\otimes \ell^2(G))$, $\Phi^\pi(k)\xi(r)=\sum_{s\in G}\pi(k(r,s))\xi(r)$, in the sense that $\Phi^\pi(k)\pi^2(f)=\pi^2(\Phi(k)f)$.
In fact, $\Phi^\pi|_{\bk_c(\cM)}$ is faithful and for all $F\in \cF$, $\Phi^\pi(\bk_F(\cM))=\bk_F^\pi(\cM)$.
Hence, $\Phi(\bk(\cM))= \bk^\pi(\cM)$ and $\Phi$ is an isometry because each restriction $\Phi^\pi|_{\bk_F(\cM)}$ is an isometry.
Since $\pi^2(\ell^2(\cM))$ is $\wot$-dense in $\pi^2(\ell^2_\wstar(\cM))$, $\Phi^\pi(\bk(\cM))\pi^2(\ell^2_\wstar(\cM))\subset \pi^2(\ell^2_\wstar(\cM))$ and it follows that $\Phi$ can be extended to a *-homomorphism $\bk(\cM)\to \bB(\ell^2_\wstar(\cM))$ that we also denote $\Phi$ and call the canonical $*$-homomorphism from $\bk(\cM)$ to $\bB(\ell^2_\wstar(\cM))$.

\begin{lemma}\label{lem: the homomorphism Psi}
    There exists a normal $*$-homomorphism $\Psi \colon \bk^\pi_\wstar(\cM) \to \bB(\ell^2_\wstar(\cM))$ such that $\Psi\circ \Phi^\pi = \Phi$.
    In other words, if we identify $\bk(\cM) = \bk^\pi(\cM)$, then $\Psi$ is the $\wstar$-continuous extension of $\Phi$.
\end{lemma}
\begin{proof}
 Since $\bk^\pi_\wstar(\cM)\pi^2(\ell^2_\wstar(\cM))\subset \pi^2(\ell^2_\wstar(\cM))$; it suffices to define $\Psi$ by $\Psi(T)f:=(\pi^2)^{-1}( T\pi^2(f) )$ ($T\in \bk^\pi_\wstar(\cM)$ and $f\in \ell^2_\wstar(\cM)$).
 Formally, for each $T\in \bk^\pi_\wstar(\cM)$, this only defines a function $\ell^2_\wstar(\cM)\to \ell^2_\wstar(\cM)$.
 It is an adjointable operator because, for all $f,g\in \ell^2_\wstar(\cM)$, we can use the fact that $\pi(\langle f,g\rangle) = \pi^2(f)^*\pi^2(g)$ to get $\pi( \langle \Psi(T)f,g\rangle ) = \pi^2(\Psi(T)f)^*\pi^2(g) = (T\pi^2(f))^*\pi^2(g) = \pi^2(f)^*T^*\pi^2(g) =\pi(\langle f,\Psi(T)^*g\rangle)$ and it follows that $\Psi(T)$ is adjointable with adjoint $\Psi(T^*)$.
 The rest of the proof is left to the reader.
\end{proof}

\begin{theorem}\label{thm: concrete description of algebra of kernels}
There exists an isomorphism of W*-algebras $\kappa^\pi\colon \bk_\wstar(\cM)\to \bk_\wstar^\pi(\cM)$ that transforms the canonical action $\beta^\wstar$ of $G$ on $\bk_\wstar(\cM)$ \cite[Section 5.4]{abadie2021amenability} into $\Ad(1\otimes \rho)$.
\end{theorem}
\begin{proof}
    By construction, $\bk_\wstar(\cM)$ is the $\wstar-$closure of the image of $\kappa\colon \bk(\cM)\to \ell^\infty(G,\bB(\ell^2_\wstar(\cM))),$ $\kappa(f)(r)=\Psi(\beta_t^{-1}(f)).$
    The unique $\wstar$-continuous extension of $\kappa$ is $\kappa^\pi\colon \bk_\wstar^\pi(\cM)\to \ell^\infty(G,\bB(\ell^2_\wstar(\cM))),$ $\kappa^\pi(T)(t)= \Psi ((1\otimes \rho_\tmu ) T(1\otimes \rho_t)).$
    It is then clear that $\kappa^\pi(\bk^\pi_\wstar(\cM))=\bk_\wstar(\cM)$ and, since $\beta^\wstar$ is implemented by the translation on $\ell^\infty(G,\bB(\ell^2_\wstar(\cM)))$, $\kappa^\pi$ is $\Ad(1\otimes \rho)-\beta^\wstar-$equivariant.

    All that remains is to prove that $\kappa^\pi$ is faithful; assume that $\kappa^\pi(T)=0.$
    From the proof of Lemma \ref{lem: the homomorphism Psi} we get that for all $r,s,t\in G;$ $a\in M_r;$ $b\in M_s$ and $x,y\in X$ 
    \begin{align*}
        0 &= \langle \pi( \langle \Psi(( 1\otimes \rho_\tmu )T(1\otimes \rho_t) )) a\delta_r,b\delta_s\rangle) x,y\rangle 
        =\langle x\otimes_\pi b\delta_s,  ( 1\otimes \rho_\tmu )T(1\otimes \rho_t) (y\otimes_\pi a\delta_r)\rangle   \\
         & = \langle \pi(b)x \otimes  \delta_s,  ( 1\otimes \rho_\tmu )T(1\otimes \rho_t) (\pi(a)y \otimes \delta_r)\rangle
          = \langle \pi(b)x \otimes \delta_{s\tmu},  T( \delta_{r\tmu}\otimes  \pi(a)y)\rangle \\
          & = \langle \pi(b)x,T_{s\tmu,r\tmu}\pi(a)y\rangle
            =\langle x,\pi(b^* k_T(s\tmu,r\tmu)a)y\rangle.
    \end{align*}
    If we take $r=e$ and $a$ as the unit of $M_e$, then we get that $b^* k(s\tmu,\tmu)=0$ for all $s,t\in G$ and $b\in M_s.$
    Hence, $b^*k_T(r,s)=0$ for all $r,s\in G$ and $b\in M_{r\smu}.$
    In particular, we can take $b=k_T(r,s)$ and it follows that $k_T=0.$
    Consequently, all the entries of $(T_{r,s})_{r,s\in G} = (\pi(k_T(r,s)))_{r,s\in G}$  are zero and this implies $T=0.$
\end{proof}

\subsection{The reduced cross-sectional W*-algebra}

By \cite[Theorem 5.7]{abadie2021amenability}, the W*-crossed product $W^*_\red(\cM)$ is isomorphic to
\begin{equation}\label{equ: Wred of M}
    W^*_{\red\pi }(\cM)=\cspn^\wstar \{ \pi(m)\otimes \lambda_t\colon t\in G,\ m\in M_t \}\sbe \bB(X\otimes \ell^2(G))
\end{equation}
and $\pi\lambda\colon \cM\to \bB(X\otimes \ell^2(G))$ defined by $\pi\lambda(m):=\pi(m)\otimes \lambda_t,$ for $m\in M_t$, is a *-representation (isometric and $\wstar$ continuous on each fiber).
Notice that $W^*_{\red\pi}(\cM)\sbe \bk^\pi_\wstar(\cM)$ because, for $m \in M_t$, $\pi\lambda(m)$ is the operator corresponding to the kernel $k_m\colon G\times G\to \cM,$ $(r,s)\mapsto m\delta_t(r\smu).$

In matrix form, for all $t\in G$ and $T\in \bB(X\otimes \ell^2(G)),$ $\Ad(1\otimes \rho_t)(  (T_{r,s})_{r,s\in G}) =(T_{rt,st})_{r,s\in G}$.
Since $k_m(rt,st)=k_m(r,s),$ $W^*_{\red\pi}(\cM)$ is contained in the W*-algebra formed by the fixed points of $\Ad(1\otimes\rho)|_{\bk_\wstar^\pi(\cM)},$ $\bk_\wstar^\pi(\cM)^G.$
To prove the reverse inclusion take $T\in \bk_\wstar^\pi(\cM)^G$ and $\xi_1,\ldots,\xi_n\in C_c(G,X).$
Let $k_T$ be the kernel of $\cM$ corresponding to $T$ ($\forall r,s\in G$  $T_{r,s}=\pi(k(r,s))$, then $k(rt,st)=k(r,s)=k(r\smu,e)$). Set $F:=\cup_{i=1}^n \supp(\xi_i)$ and $S:=\sum_{t\in FF^{-1}} \pi(k(t,e))\otimes \lambda_t\in W^*_{\red\pi}(\cM).$

The set $F_{i,j}:=\{(r,t)\colon \langle [\pi\lambda( k(t,e) )\xi_i](r),\xi_j(r)\rangle\neq 0\}$ is finite because, if $(r,t)\in F_{i,j},$ then $0\neq \langle [\pi\lambda(k(t,e))\xi_i](r),\xi_j(r)\rangle = \langle \pi(k(t,e))\xi_i(\tmu r),\xi_j(r)\rangle$ and it follows that $r,\tmu r\in F$; which yields $(r,t)\in F\times (FF^{-1}).$
For all $i,j=1,\ldots, n$, all the sums below are finite and can be transformed one into another by changing variables in $G\times G$ or by using that $F_{i,j}\sbe G\times (FF^{-1})$: 
\begin{align*}
   \langle T\xi_i,\xi_j\rangle 
   & = \sum_{r,s\in G} \langle \pi(k_T(r,s))\xi_i(s),\xi_j(r)\rangle=\sum_{r,s\in G} \langle \pi(k_T(r\smu,e))\xi_i(s),\xi_j(r)\rangle \\
   & =\sum_{r,s\in G} \langle \pi(k_T(r\smu,e))(\lambda_{r\smu}\xi_i)(r),\xi_j(r)\rangle 
   = \sum_{r,s\in G} \langle [\pi\lambda(k_T(r\smu,e))\xi_i](r),\xi_j(r)\rangle  \\
   & = \sum_{r,t\in G} \langle [\pi\lambda(k_T(t,e))\xi_i](r),\xi_j(r)\rangle  
   = \sum_{t\in FF^{-1}} \langle \pi\lambda(k_T(t,e))\xi_i,\xi_j\rangle = \langle S\xi_i,\xi_j\rangle.
\end{align*}
Recalling that $W^*_{\red\pi}(\cM)$ is $\wot$-closed, we get that $T\in W^*_{\red\pi}(\cM)$.

Adding Theorem \ref{thm: concrete description of algebra of kernels} to the preceding discussion we get

\begin{theorem}\label{thm: W subalgebra}
    The algebra $W^*_{\red\pi }(\cM)$ is the set formed by those $T\in \bk_\wstar^\pi(\cM)$ such that for all $r\in G,$ $s\mapsto T_{rs,s}$ is constant.
    In other words, $W^*_{\red}(\cM)$ is the set $\bk_\wstar(\cM)^G$ of fixed points of $\beta^\wstar$.
\end{theorem}

\begin{remark}\label{rmk:subbundles}
    We say that $\cN=\{N_t\}_{t\in G}$ is a \Wstar Fell subbundle of $\cM$ if each fiber $N_t$ is a $\wstar$-closed subspace of the respective fiber $M_t$ and $\cN$ is closed under the $*$-algebraic operations of $\cM$.
    Thinking of the elements of $\bk_\wstar(\cN)$ and $\bk_\wstar(\cM)$ as matrices, we get that $\bk_\wstar(\cN)$ is a von Neumann subalgebra of $\bk_\wstar(\cM)$ and this inclusion gives $W^*_\red(\cN)\subset W^*_\red(\cM)$.
    Moreover, if $\cN\cM\cN\subset \cN$, in which case we say $\cN$ is heredetary in $\cM$, then $\bk_\wstar(\cN)$ and $W^*_\red(\cN)$ are hereditary in $\bk_\wstar(\cM)$ and $W^*_\red(\cM)$, respectively.
\end{remark}

\begin{corollary}\label{cor:kernel and cross sectional algebra for Wstar semidirect product bundle}
    If $\cM$ is the semidirect product bundle a \Wstar dynamical system $(M,G,\gamma)$, then $\bk_\wstar(\cM)$ is \Wstar isomorphic to the von Neumann tensor product $\bB(\ell^2(G))\bar\otimes M$.
    Moreover, the isomorphism can be established in such a way that $W^*_\red(\cM)$ is mapped onto Von Neuman tensor product $M\bar\rtimes_\gamma G$.
\end{corollary}
\begin{proof}
    Say $\gamma$ is implemented by a unitary representation $U\colon G\to \bB(X)$ with $M$ a $\wstar$-subalgebra of $\bB(X)$ containing the unit.
    Then $\pi\colon \cM\to \bB(X),$ $m\delta_t\mapsto mU_t$, is a fiberwise normal and isometric representation that we can use to get $\bk_\wstar(\cM)\cong \bk_\wstar^\pi(\cM)\subset \bB(X\otimes \ell^2(G))$.
    By definition, $T\in \bk_\wstar^\pi(\cM)$ if and only if there exists a function $k_T\colon G\times G\to M$ such that $T_{r,s}=k_T(r,s)U_{r\smu}$; which we can write as $T_{r,s} = U_r\gamma_\rmu(k_T(r,s))U_\smu$.
    Consequently, for the operator unitary operator $V\in \bB(X\otimes \ell^2(G))$ such that $V(x\otimes \delta_r)=U_rx\otimes \delta_r$, we have $V (\bB(X)\bar\otimes M) V^*=\bk_\wstar^\pi(\cM)\cong \bk_\wstar(\cM)$.
    Moreover, if we consider $M\bar\rtimes_\gamma G\subset \bB(X\otimes \ell^2(G))$ as done in \cite{ADaction1979}, then $V (M\bar\rtimes_\gamma G) V^*=W^*_\red(\cM)$ because $V^*(mU_t\otimes \lambda_t)V =\Pi(m)(1\otimes\lambda_t)$ with $\Pi(m)\xi(r)=\gamma_\rmu(m)\xi(r)$.
 \end{proof}

\subsection{The W*-algebra of kernels as a reduced cross-sectional W*-algebra}\label{sec:enlarged-Fell-bundles}
We  aim to give \(\ell^\infty(G, \cM) := \{\ell^\infty(G, M_t)\}_{t \in G}\) the structure of a W*-Fell bundle.
Naturally, the norm of \(f \in \ell^\infty(G, M_r)\) is defined as \(\|f\| := \sup\{\|f(t)\| : t \in G\}\).  
Observe that if \({M_t}_*\) is the predual of \(M_t\), then \(\ell^\infty(G, M_t)\) is the dual space of \(\ell^1(G, {M_t}_*)\).  
The product and involution for \(f \in \ell^\infty(G, M_r)\) and \(g \in \ell^\infty(G, M_s)\) are given by $(fg)(t) := f(t)g(\rmu t)$ and $f^*(t) := f(rt)^*$.

\begin{example}\label{exa: ell infty of semidirect product bundle}
    If \(\cM\) is the semidirect product bundle of a \Wstar dynamical system $(M,G,\gamma)$, then \(\ell^\infty(G, \cM)\) is the semidirect product bundle of the system $(\ell^\infty(G, M),G,\tilde{\gamma})$ with \(\tilde{\gamma}_t(f)(r) = \gamma_t(f(\tmu r))\).
\end{example}

\begin{remark}
    At this point we do not know if $\ell^\infty(G,\cM)$ is a W*-Fell bundle. 
    One can check this directly, or one may use $\pi$ to construct a function $\pi^\infty\colon \ell^\infty(G,\cM)\to \bB(X\otimes \ell^2(G))$ as in the proof of Theorem \ref{thm: iso algebra of kernels}.
    That function happens to be fiberwise linear, isometric and $\wstar-\wot$ continuous on bounded sets.
    In addition, it is multiplicative and involutive.
    Using $\pi^\infty$ one proves that $\ell^\infty(G,\cM)$ is a W*-Fell bundle.
    For example, take $f,g,h\in \ell^\infty(G,\cM)$.
    Since $(fg)h$ and $f(gh)$ are in the same fiber and $\pi^\infty((fg)h)=\pi^\infty(f)\pi^\infty(g)\pi^\infty(h)=\pi^\infty(f(gh))$, we have  $(fg)h=f(gh)$.
\end{remark}

We regard \(\cM\) as a \Wstar Fell subbundle of \(\ell^\infty(G, \cM)\) by identifying \(m \in M_r\) with the constant function \( (t \mapsto m) \in \ell^\infty(G, M_r)\).
By Remark \ref{rmk:subbundles}, $W^*_\red(\cM)$ is a \Wstar subalgebra of $W^*_\red(\ell^\infty(G,\cM)).$
On the other hand, $W^*_\red(\cM)$ is also a W*-subalgebra of $\bk_\wstar(\cM)$ (Theorem \ref{thm: W subalgebra}).

\begin{theorem}[c.f. {\cite[Lemme 3.10]{ADaction1979}}]\label{thm: iso algebra of kernels}
    There exists an isomorphism $\Omega\colon W^*_\red(\ell^\infty(G,\cM))\to \bk_\wstar(\cM)$ whose restriction to $W^*_\red(\cM)$ is the identity map.
\end{theorem}
\begin{proof}
    Usign $\pi\lambda\colon \cM\to \bB(X\otimes \ell^2(G))$, $\pi\lambda(m\in M_t)=\pi(m)\otimes \lambda_t$, and Theorem \ref{thm: concrete description of algebra of kernels} we get $\bk_\wstar(\cM) = \bk^{\pi\lambda}_\wstar(\cM)\sbe \bB(X\otimes \ell^2(G)\otimes \ell^2(G)).$
    Our next goal is to identify $W^*_\red(\ell^\infty(G,\cM))$ with a von Neumann subalgebra of $\bB(X\otimes \ell^2(G)\otimes \ell^2(G)).$
    To do this we construct a *-representation $\pi^\infty\colon \ell^\infty(G,\cM)\to \bB(X\otimes \ell^2(G)).$
    For $f\in \ell^\infty(G,M_r)$ we define $\pi^\infty(f)$ by $(\pi^\infty(f)\xi)(s)=\pi(f(s))\xi(\rmu s)$ for all $s\in G$ and $\xi\in \ell^2(G,X)=X\otimes \ell^2(G).$
    Then $\pi^\infty(f)(x\otimes \delta_s) = \pi(f(rs))x\otimes\delta_{rs}$ and, for $g\in \ell^\infty(G,M_t)$,
    \begin{align*}
        \pi^\infty(f)\pi^\infty(g)(x\otimes \delta_s)
        = \pi(f(rts))\pi(g(ts))x\otimes \delta_{rts} = \pi(fg(rts))x\otimes \delta_{rts}= \pi^\infty(fg)(x\otimes \delta_t).
    \end{align*}
    Then $\pi^\infty(f)\pi^\infty(g) = \pi^\infty(fg).$
    In addition, $\pi^\infty(f^*)=\pi^\infty(f)^*$ because for all $\eta\in \ell^2(G,X)$ we have
    \begin{align*}
        \langle \pi^\infty(f)(x\otimes \delta_s),y\otimes \delta_t \rangle 
        &= \langle \pi(f(rs))x,y\rangle\langle \delta_{rs},\delta_t\rangle
        = \langle x\otimes \delta_s,\pi(f(t)^*)y\otimes \delta_{\rmu t}\rangle\\
        &=\langle x\otimes \delta_s,\pi(f^*(\rmu t))y\otimes \delta_{\rmu t}\rangle = \langle x\otimes \delta_s , \pi^\infty(f^*)(y\otimes \delta_t)\rangle.
    \end{align*}

    The identity $\langle \pi^\infty(f)\xi,\eta\rangle = \sum_{s\in G} \langle \pi(f(s))\xi(\rmu s),\eta(s)\rangle$ can be used to prove that $\pi^\infty$ is isometric on each fiber and $\wstar-\wot$ continuous when restricted to bounded sets. 
    Then, $\pi^\infty$ is $\wstar$ continuous and we may identify $W^*_\red(\ell^\infty(G,M))$ with 
    \begin{equation*}
        W^*_{\red\pi^\infty}(\ell^\infty(G,M)) =\cspn^\wstar\{  \pi^\infty(f) \otimes \lambda_t\colon  t\in G,\ f\in \ell^\infty(G,M_t)\}\sbe \bB(X\otimes \ell^2(G)\otimes \ell^2(G)).
    \end{equation*}

    Let $U\colon X\otimes \ell^2(G)\otimes \ell^2(G)\to X\otimes \ell^2(G)\otimes \ell^2(G)$ be the the bounded operator such that $U(x\otimes \delta_q\otimes \delta_s) = x\otimes \delta_{q\smu q} \otimes \delta_q$.
    Then $U$ is unitary and $U^*(x\otimes \delta_q\otimes \delta_s) = x\otimes \delta_s\otimes \delta_{sq^{-1}s}.$ 
    We claim that
    \begin{equation}\label{equ: congugate by U W}
     UW^*_{\red\pi^\infty}(\ell^\infty(G,M)) U^* = \bk_\wstar^{\pi\lambda}(\cM).   
    \end{equation}
    To prove the inclusion $\sbe $  we take $T\in W^*_{\red\pi^\infty}(\ell^\infty(G,M))$ and let $f_T$ be the cross section of $\ell^\infty(G,\cM)$ such that $T_{r,s} = \pi^\infty(f_T(r\smu)).$
    For all $r,s,p,q\in G$ and $x,y\in X$ we have
    \begin{align*}
     \langle (U T U^*)_{r,s} (y\otimes \delta_q),x\otimes \delta_p\rangle& =     \langle U T U^* (y \otimes \delta_q\otimes \delta_s ),x \otimes \delta_p\otimes \delta_r \rangle
        \\
        & = \langle T (y\otimes \delta_s \otimes \delta_{sq^{-1}s} ),x\otimes\delta_r\otimes  \delta_{rp^{-1}r} \rangle \\
        & =\langle T_{rp^{-1}r,sq^{-1}s} ( y\otimes \delta_s),x\otimes \delta_r \rangle \\
        & = \langle \pi^\infty( f_T( rp^{-1}r\smu q \smu ) ) (y\otimes \delta_s),x\otimes \delta_r \rangle\\
        & =\langle \pi( f_T( rp^{-1}r\smu q \smu  )(rp^{-1}r\smu q) )y\otimes \delta_{rp^{-1}r\smu q},x\otimes \delta_r\rangle\\
        & =\langle \pi( f_T( rp^{-1}r\smu q \smu  )(rp^{-1}r\smu q) )y,x\rangle \langle \delta_{rp^{-1}r\smu q},\delta_r\rangle\\
        & =\langle \pi( f_T( r\smu  )(r) )y,x\rangle \langle \delta_{p^{-1}r\smu q},\delta_e\rangle\\
        & =\langle \pi( f_T( r\smu  )(r) )y,x\rangle \langle\lambda_{r\smu}\delta_q,\delta_p\rangle\\
        & = \langle [\pi( f_T( r\smu  )(r))\otimes \lambda_{r\smu}] (y\otimes \delta_q),x\otimes \delta_p\rangle\\
        & = \langle \pi\lambda(f_T( r\smu  )(r)) (\delta_q\otimes y),\delta_p\otimes x\rangle.
    \end{align*}
    Then $(U T U^*)_{r,s} = \pi\lambda(f_T( r\smu  )(r))$ and it follows that $UTU^*\in \bk^{\pi\lambda}_\wstar(\cM).$

    Given $k\in \bk_c(\cM)$, define the cross section $f\colon G\to \ell^\infty(G,\cM)$ by $f(z)(w) = k(w,z^{-1}w).$
    Then $f$ has finite support and for $T:=\sum_{t\in G}\pi^\infty(f(t))\otimes \lambda_t$ we have $f=f_T$.
    Moreover, $(UTU^*)_{r,s}=\pi\lambda( f(r\smu)(r) ) = \pi\lambda(k(r,s))$ and it follows that $\bk^{\pi\lambda}(\cM)\sbe  UW^*_{\red\pi^\infty}(\ell^\infty(G,M)) U^*.$   
    Recalling that $\bk^{\pi\lambda}(\cM)$ is $\wstar$-dense in $\bk^{\pi\lambda}_\wstar(\cM)$ we obtain  \eqref{equ: congugate by U W}.

    Assume that $T\in W^*_{\red}(\cM)\sbe W^*_{\red\pi^\infty}(\ell^\infty(G,M)),$ that is to say that for all $t\in G,$ $f_T(t)\in \ell^\infty(G,M_t)$ is constant.
    Then $(UTU^*)_{rt,st} = \pi\lambda(f_T(rt(st)^{-1})(rt))=\pi\lambda(f_T(rs^{-1})(r))=(UTU^*)_{r,s}$.
    In particular, for $t\in G$ and $m\in M_t$, $U(\pi^\infty(m)\otimes \lambda_t)=\pi\lambda(m)\otimes \lambda_t$ and it follows that, up to faithful normal representations, we may define $\Omega$ as conjugation by $U.$
\end{proof}

\subsection{The W*-partial cross-sectional Fell bundle}\label{sec:partial-Fellbundles}

It is a known fact that given a group $G$ and a normal subgroup $H\trianglelefteq G,$ $G$ is amenable if and only if both $H$ and $G/H$ are amenable.
We want a similar result for W*-Fell bundles.
More precisely, we want to construct a W*-Fell bundle $\dot \cM$ over $\dot G :=G/H$ such that $\cM$ is W$^*$-amenable if and only if both $\dot \cM$ and the reduction $\cM_H:=\{M_t\}_{t\in H}$ are W$^*$-amenable ($\cM_H$ is a W*-Fell bundle with the structure inherited from $\cM$).
For $t\in G$ we set $\dot t \equiv tH.$

We fix a subgroup $H\sbe G$ and define the natural map $\lambda^\cM\colon \cM \to W^*_\red(\cM)$ as the function such that, if we identify $W^*_\red(\cM)=W^*_{\red\pi}(\cM)$ as in \eqref{equ: Wred of M}, then $\lambda^\cM(m) =\pi(m)\otimes \lambda_t$ for $m\in M_t$.

\begin{proposition}
    There exists a *-homomorphism $\mu\colon W^*_\red(\cM_H)\to W^*_\red(\cM)$ such that $\mu\circ \lambda^{\cM_H}=\lambda^{\cM}|_{\cM_H}.$
    Moreover, $\mu(W^*_\red(\cM_H))$ is a W*-subalgebra and $\mu$ is a \Wstar isomorphism over its image.
\end{proposition}
\begin{proof}
    We have $W^*_\red(\cM)=W^*_{\red\pi}(\cM)\sbe \bB(\ell^2(G,X))$ and $W^*_\red(\cM_H)=W^*_{\red\pi|_{\cM_H}}(\cM_H)\sbe \bB(X\otimes \ell^2(H))$.
    Fix a complete set of representatives of the cosets,  $\{z_u\}_{u\in \dot G}$,  such that $z_H=e$.
    For each $u\in \dot G $ we have a unitary operator $V_u \colon X\otimes\ell^2(H)\to X\otimes\ell^2(u)\sbe X\otimes \ell^2(G)$ such that $V_u\xi(sr_u)=\xi(s).$
    Using the decomposition $X\otimes \ell^2(G)=\bigoplus_{u\in \dot G }X\otimes \ell^2(u)$ we can define a faithful and normal map
    \begin{align*}
        \kappa& \colon \bB(X\otimes \ell^2(H))\to \bB(X\otimes \ell^2(G)) & T&\mapsto \bigoplus_{u\in \dot G } V_u T{V_u}^*.
    \end{align*}
    It suffices to set $\mu:=\kappa|_{W^*_\red(\cM_H)}$ because for all $t\in H$ and $m\in M_t,$ $\kappa(\pi(M)\otimes \lambda^H_t)=\pi(m)\otimes \lambda^G_t$.
\end{proof}

From now on we view $W^*_\red(\cM_H)\sbe W^*_\red(\cM)$ via the map $\mu$ provided by the proposition above.
With this convention there is no problem on defining, for each coset $u\in \dot G $,
\begin{equation*}
  \dot{M}_u\equiv  W^*_\red(\cM_u):=\cspn^\wstar \{ \lambda^\cM(m)\colon t\in u, m\in M_t  \}\sbe W^*_{\red}(\cM).
\end{equation*}

\begin{proposition}\label{prop: the partial cross sectional w bundle}
    If $H$ is normal in $G$, then $\dot{\cM}:=\{ \dot{M}_u \}_{u\in \dot G }$ is a W*-Fell bundle with the structure (norm, product, involution and $\wstar$-topology) inherited from $W^*_\red(\cM).$
\end{proposition}
\begin{proof}
    Basically, it suffices to show that for all $u,v\in \dot G ,$ ${\dot{M}_u}^* \sbe \dot{M}_{u^{-1}}$ and $\dot{M}_u\dot{M}_v\sbe \dot{M}_{uv}$.
\end{proof}

\begin{definition}[{c.f.\cite[VIII 6]{FlDr88}}]\label{Def:Partial-Cross-Sec}
    For a normal subgroup $H\trianglelefteq G$, the \Wstar{}partial cross-sectional bundle over $\dot G $ derived from $\cM$ is $\dot{\cM}.$
\end{definition}

\begin{theorem}[c.f. Theorem \ref{thm: W subalgebra}]\label{thm: identificacion of kernels of WMH}
    Let $H$ be a normal subgroup of $G$ and write $\bk_\wstar(\cM)^H$ for the W*-algebra of fixed points of $\beta^\wstar|_H$. Then there exists a W*-isomorphism 
    $$\mu\colon \bk_\wstar(\dot{\cM})\congto \bk_\wstar(\cM)^H$$ such that $\mu(W^*_\red(\dot{\cM}))=W^*_\red(\cM).$
\end{theorem}
\begin{proof}
    We think $W^*_\red(\cM)=W^*_{\red\pi}(\cM)\sbe \bB(X\otimes \ell^2(G))$ and this gives a canonical representation $\iota\colon \dot{\cM}\to \bB(X\otimes \ell^2(G) ),$ $\iota(f)=f.$
    Then 
    \begin{equation*}
     \bk_\wstar(\dot{\cM})=\bk^\iota_\wstar(\dot{\cM})\sbe \bB(X\otimes \ell^2(G) \otimes \ell^2(\dot G )).   
    \end{equation*}

For each $u\in \dot G $ we let $p_u\in \bB(X\otimes \ell^2(G) )$ be the orthogonal projection with image $X\otimes \ell^2(u).$
Define $U\colon X\otimes \ell^2(G) \to X\otimes \ell^2(G)\otimes  \ell^2(\dot G ) $ by $U\xi = \sum_{u\in \dot G } p_u\xi\otimes \delta_u .$
With the identification $X\otimes \ell^2(G)\otimes  \ell^2(\dot G )=\ell^2(\dot G, X\otimes \ell^2(G)),$ $U^*\eta = \sum_{u\in \dot G } p_u \eta(u)$.
Hence, $UU^*\eta = \sum_{v\in \dot G}p_v (\sum_{u\in \dot G } p_u \eta(u))\otimes \delta_v = \sum_{v\in \dot G}p_v \eta(v)\otimes \delta_v.$

It suffices to show that 
\begin{equation*}
    \mu\colon \bk^\iota_\wstar(\dot{\cM})\to \bB(X\otimes \ell^2(G) ),\qquad \mu(T)=U^*TU
\end{equation*}
is a faithful *-homomorphism with image $\bk_\wstar^\pi(\cM)^H$.

Clearly, $\mu$ is linear and preserves the adjoints.
To prove $\mu$ is multiplicative it suffices to show that for all $T\in \bk^\iota_\wstar(\dot{\cM})$ we have $UU^*TU=TU$.
For all $\xi\in X\otimes \ell^2(G) $ and $\eta\in X\otimes \ell^2(G)\otimes \ell^2(\dot G )$
\begin{align*}
    \langle UU^*TU\xi,\eta\rangle 
    &= \langle TU\xi,UU^*\eta\rangle 
    =\sum_{u,v\in \dot G }\langle T  (p_v \xi \otimes\delta_v),p_u\eta(u)\otimes \delta_u\rangle
    =\sum_{u,v\in \dot G }\langle p_uT_{u,v} p_v \xi ,\eta(u)\rangle
\end{align*}

We claim that $\langle p_uT_{u,v}p_v\xi,\eta(u)\rangle =\langle T_{u,v}p_v\xi,\eta(u)\rangle$.
Indeed, by construction $T_{u,v}\in \iota(\dot{M}_{uv^{-1}})\sbe W^*_{\red\pi}(\cM),$ then there exists a cross-section $a_{u,v}^T$ of $\cM$ with $T_{u,v} = (\pi(a_{u,v}^T(r\smu)))_{r,s\in G}$ and $a_{u,v}^T(t)=0$ if $t\notin uv^{-1}.$
We have $\langle p_uT_{u,v}p_v \xi,\eta(u)\rangle = \sum_{r\in u,s\in v}\langle \pi(a_{u,v}^T(r\smu))\xi(s),\eta(r)\rangle.$
If $s\in v$ and $r\notin u,$ then $a_{u,v}^T(r\smu)=0$ because, in other case, $r\smu\in uv^{-1}$ and this yields $r\in u.$
So,
\[ \langle p_uT_{u,v}p_v \xi,\eta\rangle=\sum_{r\in G,s\in v}\langle \pi(a_{u,v}^T(r\smu))\xi(s),\eta(r)\rangle=\langle T_{u,v}p_v\xi,\eta(r)\rangle\]
and we get $\langle UU^*TU\xi,\eta\rangle =\sum_{u\in \dot G }\sum_{v\in \dot G }\langle T_{u,v}p_v\xi,\eta(u)\rangle
    =\langle T(\sum_{v\in \dot G }p_v\xi\otimes\delta_v),\eta \rangle =\langle TU\xi,\eta\rangle.$
Consequently, for all $S,T\in \bk_\wstar(\dot{\cM}),$ $\mu(S)\mu(T)=U^*SUU^*TU=USTU^*=\mu(ST).$

At this point we know $\mu$ is a $\wstar$-continuous *-homomorphism. 
To prove that $\mu(T)\in \bk_\wstar(\cM)^H$ and that $\mu$ is faithful we compute $\mu(T)_{r,s}.$
For $\xi,\eta\in X$ and $r,s\in G$
\begin{align*}
    \langle \mu(T)_{r,s}\xi,\eta\rangle 
     &= \langle \mu(T)(\xi\otimes  \delta_s ),\eta\otimes  \delta_r \rangle
     = \langle T(\xi\otimes  \delta_s\otimes \delta_{sH} ),\eta\otimes  \delta_r\otimes \delta_{rH} \rangle
       = \langle T_{rH,sH}( \xi\otimes  \delta_s ),\eta\otimes  \delta_r \rangle.     
\end{align*}

Assume that $\mu(T)=0$ and fix $u,v\in \dot{G}$.
For all $r,s\in G$ and $x,y\in X$ we have
\[ \langle T_{u,v}(y\otimes \delta_s),x\otimes\delta_r\rangle = \langle \pi(a_{u,v}^T(r\smu))y,x\rangle.\]
If $r\smu\notin uv^{-1},$ the $a_{u,v}^T(r\smu)=0$ and $\langle T_{u,v}(y\otimes \delta_s),x\otimes\delta_r\rangle=0$.
If $r\smu\in uv^{-1},$ then $r\smu = pq^{-1}$ for some $p\in u$ and $q\in v.$
In this case, $\langle T_{u,v}(y\otimes \delta_s),x\otimes\delta_r\rangle = \langle \pi(a_{u,v}^T(pq^{-1}))y,x\rangle = \langle T_{u,v}(y\otimes \delta_q),x\otimes\delta_p\rangle = \langle \mu(T)_{p,q}(y\otimes \delta_q),x\otimes\delta_p\rangle=0.$
Then, for all $u,v\in \dot G$; $r,s\in G$ and $x,y\in X$, $\langle T_{u,v}(y\otimes \delta_s),x\otimes\delta_r\rangle=0$, so $T=0$.

For all $r,s\in G$ and $t\in H$ we have 
\begin{align*}
   \langle \mu(T)_{rt,st}\xi,\eta\rangle 
   &= \langle T_{rtH,stH}(\xi\otimes \delta_{st}),\eta\otimes \delta_{rt}\rangle 
= \langle T_{rH,sH}(\xi\otimes \delta_{st}),\eta\otimes \delta_{rt}\rangle
= \langle \pi( a_{rH,sH}^T(rt (st)^{-1}))\xi,\eta\rangle\\
&  = \langle \pi( a_{rH,sH}^T(rs^{-1}))\xi,\eta\rangle
=\langle \mu(T_{r,s}\xi,\eta\rangle.
\end{align*}
Hence, $\mu(T)_{rt,st}=\mu(T)_{r,s}$ and it follows that $\mu(T)\in \bk^{\pi}_\wstar(\cM)^H$.

The final step is to prove that $\mu$ is surjective, to do this we take $T\in \bk_\wstar^\pi(\cM)^H$ with kernel $k\colon G\times G\to \cM$ and construct $S\in \bk_\wstar^\iota(\dot{\cM})$ such that $\mu(S)=T.$
Fix a complete set of representatives $\{ p_u \}_{u\in \dot G }.$
We claim that for each $(u,v)\in \dot G \times \dot G $ there exists a function 
\begin{equation*}
    k_{u,v}\colon G\times G\to \cM \qquad k_{u,v}(r,s)=\begin{cases}
        k(p_u h,p_vk)\mbox{ if }(r,s)=(p_uht,p_vkt) \mbox{ for some }h,k\in H, \ t\in G\\
        0\in M_{r\smu}\mbox{ otherwise}
    \end{cases}.
\end{equation*}
To prove $k$ is defined we assume $(p_uh_1t_1,p_vk_1t_1)=(p_uh_2t_2,p_vk_2t_2)$ with $h_1,h_2,k_1,k_2\in H$ and $t_1,t_2\in G.$
Then ${h_1}^{-1}h_2={k_1}^{-1}k_2= t_1{t_2}^{-1}\in H$ and the fact that $T$ is a $\beta^\wstar|_H$-fixed point implies
\begin{equation*}
    k_{u,v}(p_u h_1,p_vk_1) = k_{u,v}(p_u h_1 t_1{t_2}^{-1} ,p_vk_1 t_1{t_2}^{-1})=k(p_uh_2,p_vk_2)\in M_{(p_u h_2t_2)(p_vk_2t_2)^{-1}}=M_{r\smu},
\end{equation*}
so $k$ is defined and it is a kernel.

To prove the existence of an operator $S_{u,v}\in \bk^\pi_\wstar(\cM)$ with kernel $k_{u,v}$ it suffices to show that for all $\xi,\eta\in C_c(G,X),$
\begin{equation*}
    |\sum_{r,s\in G} \langle \pi(k_{u,v}(r,s))\xi(s),\eta(r)\rangle|\leq \| T \|\|\xi\|\|\eta\|.
\end{equation*}

To eliminate null terms of the sum above we consider only the pairs $(r,s)=(p_uht,p_vkt)$  with $h,k\in H$ and $t\in G.$
To avoid repetitions we set $t=p_w$ for $w\in \dot G .$
Defining $\eta_{u,w}\in X \otimes \ell^2(u)$ by $\eta_{u,v}(z)=\eta(zp_w)$ we get
\begin{align*}
    \sum_{r,s\in G} \langle \pi(k_{u,v}(r,s))\xi(s),\eta(r)\rangle
    & = \sum_{w\in \dot G } \sum_{h,k\in H} \langle T_{p_uh,p_vk} \xi(p_vkp_w),\eta(p_uhp_w)\rangle\\
    & = \sum_{w\in \dot G } \langle T \xi_{v,w},\eta_{u,w}\rangle.
\end{align*}
Using the fact that $X\otimes \ell^2(G) $ is the direct sum of the orthogonal family $\{\ell^2(u)\otimes X\colon u\in \dot G \}$ we obtain
\begin{align*}
    |\sum_{r,s\in G} \langle \pi(k_{u,v}(r,s))\xi(s),\eta(r)\rangle|
    & \leq \|T\|  (\sum_{w\in \dot G }\|\xi_{v,w}\|^{2})^{1/2}(\sum_{w\in \dot G }\|\eta_{v,w}\|^{2})^{1/2}\\
    & \leq \|T\|  (\sum_{w\in \dot G }\sum_{h\in H}\|\xi(p_uhp_w)\|^{2})^{1/2}(\sum_{w\in \dot G }\|\eta_{v,w}\|^{2})^{1/2}\\
    & \leq \|T\|  (\sum_{t\in G}\|\xi(p_ut)\|^{2})^{1/2}(\sum_{w\in \dot G }\|\eta_{v,w}\|^{2})^{1/2}
    \leq \|T\|\|\xi\|\|\eta\|.
\end{align*}

To prove that $S_{u,v}\in W^*_{\red\pi}(\cM)$ we use the characterization of $W^*_{\red\pi}(\cM)$ as the fixed points of $\bk_\wstar^\pi(\cM)$ under the canonical action of $G.$
If $(r,s)=(p_uhl,p_vkl)$ with $h,k\in H$ and $l\in G,$ then for all $t\in G$ we have $(rt,st)=(p_uhlt,p_vklt)$ and $k_{u,v}(rt,st)=k(p_uh,p_vk)=k_{u,v}(r,s).$
If $k_{u,v}(r,s)=0,$ it can not happen that $k_{u,v}(rt,st)\neq 0$ for, in that case, what we proved in the preceding sentence would imply that $0\neq k_{u,v}(rt,st)=k_{u,v}(rt\tmu,st\tmu)=k_{u,v}(r,s).$

Now we show there exists an operator $S\in \bk_\wstar^\iota(\dot{\cM})$ with kernel $(r,s)\mapsto S_{u,v}.$
By construction, $\langle S_{u,v}\xi,\eta\rangle = \sum_{w\in \dot G }\langle T\xi_{v,w},\eta_{u,v}\rangle.$
For all $\xi,\eta\in X\otimes \ell^2(G)\otimes  \ell^2(\dot G ) $ we have
\begin{align*}
   | \sum_{u,v\in \dot G }\langle S_{u,v}\xi(v) & ,\eta(u)\rangle |
   =|\sum_{u,v,w\in \dot G }\langle T (\xi(v))_{v,w},(\eta(u))_{u,w} \rangle | \\
   & =|\sum_{w\in \dot G }\langle T\sum_{v\in \dot G }(\xi(v))_{v,w},\sum_{u\in \dot G }(\eta(u))_{u,w} \rangle | \\
   &\leq \|T\| \left(\sum_{w\in \dot G }\| \sum_{v\in \dot G } (\xi(v))_{v,w} \|^2 \right)^{1/2}\left(\sum_{w\in \dot G }\| \sum_{u\in \dot G } (\eta(u))_{u,w} \|^2 \right)^{1/2}\\
   &\leq \|T\| \left(\sum_{w\in \dot G } \sum_{v\in \dot G } \|(\xi(v))_{v,w} \|^2 \right)^{1/2}\left(\sum_{w\in \dot G } \sum_{u\in \dot G } \|(\eta(u))_{u,w} \|^2 \right)^{1/2}\\
   &\leq \|T\| \left(\sum_{w\in \dot G } \sum_{v\in \dot G } \sum_{h\in H}\|\xi(v)(p_vhp_w)\|^2 \right)^{1/2}\left(\sum_{w\in \dot G } \sum_{u\in \dot G } \|(\eta(u))_{u,w} \|^2 \right)^{1/2}\\
   &\leq \|T\| \left(\sum_{v\in \dot G } \|\xi(v)\|^2 \right)^{1/2}\left( \sum_{u\in \dot G } \|\eta(u)\|^2 \right)^{1/2}=\|T\|\|\xi\|\|\eta\|.
\end{align*}
This implies the existence of the operator $S\in \bk_\wstar^\iota(\dot{\cM})$ with matrix $(S_{u,v})_{u,v\in \dot G }.$

Suppose $T\in W^*_\red(\cM),$ to prove $S\in W^*_{\red}(\dot{\cM})$ it suffices to show that for all $u,v,w\in \dot G $, $k_{uw,vw}=k_{u,v}.$
Take $r,s\in G.$
If $k_{u,v}(r,s)\neq 0$, then there exists $h,k\in H$ and $t\in G$ such that $(r,s)=(p_u ht,p_vkt).$
Hence, 
\begin{align*}
 (r,s) & =(p_u p_w {p_w}^{-1}hp_w {p_w}^{-1}t,p_vp_w{p_w}^{-1}kp_w{p_w}^{-1}t)   \\
& = (p_{uw} \underbrace{{p_{uw}}^{-1}p_u p_w}_{\in H} \underbrace{{p_w}^{-1}hp_w}_{\in H} {p_w}^{-1}t,p_{vw} \underbrace{{p_{vw}}^{-1}p_vp_w}_{\in H}\underbrace{{p_w}^{-1}kp_w}_{\in H}{p_w}^{-1}t)
\end{align*}
and $k_{uw,vw}(r,s)=k(p_uhp_w,p_vkp_w)=k(p_uh,p_wk)=k_{u,v}(r,s).$
The condition $k_{u,v}(r,s)= 0$ implies $k_{uw,vw}(r,s)=0$ for, in other case, what we just proved implies $0=k_{u,v}(r,s) = k_{uww^{-1},vww^{-1}}(r,s)=k_{uw,vw}(r,s)\neq 0.$

The final part of the proof consists on showing that $\mu(S)=T.$
By construction, 
\begin{align*}
    \langle \mu(S)(\xi\otimes \delta_r),\eta\otimes\delta_s \rangle
    &= \langle (S_{sH,rH})_{s,r} \xi,\eta\rangle 
     = \langle \pi(k_{sH,rH}(r,s))\xi,\eta\rangle = \langle \pi(k(r,s))\xi,\eta\rangle\\
     &=\langle T (\xi\otimes \delta_r),\eta\otimes\delta_s \rangle,
\end{align*}
where the second to last equality follows from the facts that $(r,s)=(p_{rH}{p_{rH}}^{-1}r,p_{sH}{p_{sH}}^{-1}s)$ and ${p_{rH}}^{-1}r,{p_{sH}}^{-1}s\in H$.
\end{proof}

\subsubsection{An example: semidirect product bundles}
Take a W*-dynamical system $(M,G,\gamma)$, a normal subgroup $H\trianglelefteq G$ and let $\cM=\{M\delta_t\}_{t\in G}$ be the semidirect product bundle of $\gamma$.
The reduction $\cM_H$ is the semidirect product of $\gamma|_H$ and the von Neumann crossed product $M\bar{\rtimes}_{\gamma|_H}H$ is the fiber over $H$ of $\dot{\cM}.$
Even though $\dot{\cM}$ is not the semidirect product of an action, it induces an action $\delta$ of $\dot G $ on $Z(M\bar{\rtimes}_{\gamma|_H}H)$ and $\dot{\cM}$ is W$^*$-amenable if and only if $\delta$ is W$^*$-amenable \cite{abadie2021amenability}.

To the best of our knowledge, it has not been proved that $\gamma$ is W$^*$-amenable if and only if $\delta$ and $\gamma|_H$ are amenable (this a consequence of Theorem \ref{thm: M amena iff N and Mh amenable}).
The closest result is \cite[Proposition 2.8]{ADactionII1982}, which implies that if $\gamma|_H$ is W$^*$-amenable and $\dot G $ is amenable, then $\gamma$ is W$^*$-amenable.

\section{Amenability, fixed points and conditional expectations}\label{sec: amenability and fixed points}

In this section we continue working with the bundle $\cM$ and representation $\pi$ of Section \ref{sec: WFell bundles} and add a subgroup $H\sbe G$.
The central partial action of $G$ on $Z:=Z(M_e)$ induced by $\cM$, $\sigma=(\{Z_t\}_{t\in G},\{\sigma_t\}_{t\in G})$, was defined in \cite{abadie2021amenability}.
Recall that $Z_t=Z(\cspn^\wstar(M_tM_\tmu))$ and that, for all $t\in G$, $m\in M_t$ and $n\in Z_\tmu,$ $\sigma_t(n)m = m n.$
The W*-enveloping action of $\sigma$ is isomorphic to $\beta^\wstar|_{Z(\bk_\wstar(\cM))}$ \cite[Theorem 5.14]{abadie2021amenability}.
By definition, $\cM$ is W$^*$-amenable if and only if $\beta^\wstar|_{Z(\bk_\wstar(\cM))}$ is a W$^*$-amenable (i.e. $\sigma$ is W$^*$-amenable).

\begin{theorem}\label{thm: Mh amenble cond exp to fixed point algebra}
    If the reduction $\cM_H=\{M_t\}_{t\in H}$ is W$^*$-amenable, then there exists a conditional expectation from $\bk_\wstar(\cM)$ to $\bk_\wstar(\cM)^H.$
\end{theorem}
\begin{proof}
    Clearly, the restriction $\sigma|_H:=(\{Z_t\}_{t\in H},\{\sigma_t\}_{t\in H})$ of $\sigma$ is the central partial action of $\cM_H$.
    Let $\tau$ and $N$ be the W*-enveloping action and algebra, respectively, of $\sigma$, see Proposition 2.7 and Remark 2.9 of \cite{abadie2021amenability}.
    We regard $Z$ as a W*-ideal of $N$ and denote $p$ the unit of $Z.$
    The $\wstar$-closure $\tilde{Z}$ of the *-ideal $\sum_{t\in H}\tau_t(Z)$ of $N$ is $\tau|_H$ invariant, so $\tilde{\tau}:=\tau|_H|_{\tilde{Z}}$ is a global W*-action of $H$.
    Moreover, $\tilde{\tau}$ is the W*-enveloping action of $\sigma|_H$ because, for all $t\in H$ and $z\in Z_\tmu$, we have $Z\cap \tilde{\tau}_t(Z)=Z\cap \tau_t(Z)=Z_t$ and $\tilde{\tau}_t(z)=\tau_t(z)=\sigma_t(z).$
    In addition, the $\wstar$-closure of $\sum_{t\in H}\tilde{\tau}_t(Z)=\sum_{t\in H}\tau_t(Z)$ is $\tilde{Z}$.
    The unit of $\tilde{Z}$ will be denoted $\tilde{p}.$

    By \cite[Definition 5.16]{abadie2021amenability}, $\tilde{\tau}$ is amenable and  \cite[Théorème 3.3]{ADactionII1982} gives a net of finitely supported functions $\{f_i\colon H\to \tilde{Z}\}_{i\in H}$ such that: (a) for all $i\in I,$ $\sum_{t\in H}f_i(t)^*f_i(t)\leq \tilde{p}$ and (b) for all $s\in G$, $\wstar \lim_i \sum_{t\in G} f_i(t)\tilde{\tau}_s(f_i(\smu t))=\tilde{p}$.
    Moreover, \cite[Lemme 3.2]{ADsystemes1987} implies that we may assume all the $f_i$'s take positive values.

    Fix a complete set of representatives $\{q_u\}_{u\in \dot G }$ for the cosets of $\dot G$.
    For $t\in G$ we set $q_t:=q_{tH}$ and $q_t^{-1}\equiv q_{tH}^{-1}$.
    We claim there is a linear and ccp function (that we express in matrix form)
    \begin{align}
        P_i\colon  \bB(X\otimes \ell^2(G) )&\to \bB(X\otimes \ell^2(G) ) \label{equ: Pi domain codomain}\\ (T_{r,s})_{r,s\in G}&\mapsto \left({}_iT_{r,s}:=\sum_{h\in H} \pi(p\tau_{q_r}(f_i(q_r^{-1}r h))^*)T_{rh,sh}\pi(\tau_{q_s}(f_i(q_{s}^{-1}s h))p)\right)_{r,s\in G}.\label{equ: Pi formula}
    \end{align}
    To prove this we define, for each $h\in H,$ the linear operator $V_h\colon C_c(G,X)\to C_c(G,X)$ by $V_h\xi(r)=\pi(\tau_{q_r}(f_i(q_r^{-1}r))p)\xi(rh^{-1}).$
    Since $\|f_i(q_r^{-1}r)\|\leq \|\sum_{k\in H}f_i(k)^*f_i(k)\|^{1/2}\leq \|p\|,$ $V_h$ is a linear contraction.
    Denoting  $1_N$ the unit of $N$, we have
    \begin{align*}
        \sum_{h\in H}\|V_h\xi\|^2
        &=\sum_{h\in H}\sum_{r\in G} \langle \pi(\tau_{q_r}(f_i(q_r^{-1}r))p)\xi(rh^{-1}),\pi(\tau_{q_r}(f_i(q_r^{-1}r))p)\xi(rh^{-1})\rangle\\
        &=\sum_{h\in H}\sum_{r\in G} \langle \pi(\tau_{q_{rh}}(f_i(q_{rh}^{-1}rh))p)\xi(r),\pi(\tau_{q_{rh}}(f_i(q_{rh}^{-1}rh))p)\xi(r)\rangle\\
        &=\sum_{r\in G}\sum_{h\in H} \langle \pi(\tau_{q_{r}}(f_i(q_{r}^{-1}rh))p)\xi(r),\pi(\tau_{q_{r}}(f_i(q_{r}^{-1}rh))p)\xi(r)\rangle\\
        &=\sum_{r\in G}\sum_{h\in H} \langle \pi(p\tau_{q_{r}}(f_i(h)^*f_i(h))p)\xi(r),\xi(r)\rangle\\
        &\leq \sum_{r\in G}\langle \pi(p\tau_{q_r}(\tilde{p})p)\xi(r),\xi(r)\rangle
        \leq \sum_{r\in G}\langle \pi(p1_Np)\xi(r),\xi(r)\rangle
        =\|\xi\|^2.
    \end{align*}
    
    For all $\xi,\eta\in C_c(G,X)$ we have
    \begin{align*}
        \sum_{r,s\in G}\langle {}_iT_{r,s} \xi(s),\eta(r)\rangle
        & = \sum_{h\in H}\sum_{r,s\in G} \langle  \pi(p\tau_{q_r}(f_i(q_r^{-1}r h))^*)T_{rh,sh}\pi(\tau_{q_s}(f_i(q_{s}^{-1}s h))p) \xi(s),\eta(r)\rangle\\
        & = \sum_{h\in H}\sum_{r,s\in G} \langle  \pi(p\tau_{q_{rh^{-1}}}(f_i(q_{rh^{-1}}^{-1}r))^*)T_{r,s}\pi(\tau_{q_{sh^{-1}}}(f_i(q_{sh^{-1}}^{-1}s))p) \xi(sh^{-1}),\eta(rh^{-1})\rangle\\
        & = \sum_{h\in H}\sum_{r,s\in G} \langle  T_{r,s}\pi(\tau_{q_s}(f_i(q_s^{-1}s))p) \xi(sh^{-1}),\pi(\tau_{q_r}(f_i(q_r^{-1}r))p)\eta(rh^{-1})\rangle\\
        & = \sum_{h\in H}\sum_{r,s\in G} \langle  T_{r,s}V_h\xi(s),V_h\eta(r)\rangle
         = \sum_{h\in H}\langle  TV_h\xi,V_h\eta\rangle
          \leq \|T\|\sum_{h\in H}\|V_h\xi\|\|V_h\eta\|\\
        & \leq \|T\|\left(\sum_{h\in H}\|V_h\xi\|^2\right)^{1/2}\left(\sum_{h\in H}\|V_h\eta\|^2\right)^{1/2} \leq \|T\|\|\xi\|\|\eta\|.
    \end{align*}
    In the first inequality, we are implicitly saying that $h\mapsto \langle  TV_h\xi,V_h\eta\rangle$ is integrable with respect to the counting measure because $|\langle  TV_h\xi,V_h\eta\rangle|\leq \|T\|\|V_h\xi\|\|V_h\eta\|$ and $\sum_{h\in H}\|V_h\xi\|^2+\sum_{h\in H}\|V_h\eta\|^2\leq \|\xi\|^2+\|\eta\|^2$.
    All the computations above imply the existence of the linear contraction $P_i$ of \eqref{equ: Pi domain codomain} satisfying \eqref{equ: Pi formula}.
    Moreover, $\langle P_i(T)\xi,\eta\rangle = \sum_{h\in H}\langle TV_h\xi,V_h\eta\rangle$ ($\xi,\eta\in C_c(G,X)$) and the complete positivity of $P_i$ follows from the fact that for all $T_1,\ldots,T_n\in \bB(X\otimes \ell^2(G) )$ and $\xi_1,\ldots,\xi_n\in C_c(G,X)$
    \begin{equation*}
        \sum_{j,k=1}^n \langle P_i({T_j}^*T_k)\xi_k,\xi_j\rangle
        = \sum_{h\in H} \sum_{j,k=1}^n\langle {T_j}^*T_k V_h\xi_k,V_h\xi_j\rangle
        = \sum_{h\in H} \langle \sum_{k=1}^nT_k V_h\xi_k,\sum_{j=1}^nT_jV_h\xi_j\rangle\geq 0.
    \end{equation*}

    Take  $T\in \bk_\wstar^\pi(\cM)$ and denote $k_T$ the kernel of $T$, i.e. $T_{r,s}=\pi(k_T(r,s))$.
    For all $r,s\in G$ we have
    \begin{equation*}
        P_i(T)_{r,s} =   \pi\left(\sum_{h\in H}p\tau_{q_r}(f_i(q_r^{-1}r h))^*k_T(r,s)\tau_{q_s}(f_i(q_{s}^{-1}s h))p\right).
    \end{equation*}
    Since $\tau_{q_s}(f_i(q_{s}^{-1}s h))p\in M_e$ and $k_T(r,s)\in M_{r\smu}$, we have 
    \begin{equation*}
     \sum_{h\in H}p\tau_{q_r}(f_i(q_r^{-1}r h))^*k_T(r,s)\tau_{q_s}(f_i(q_{s}^{-1}s h))p\in M_{r\smu}   
    \end{equation*}
    and it follows that $P_i(T)_{r,s}\in \pi(M_{r\smu})$.
    Hence, $P_i(T)\in \bk_\wstar^\pi(\cM)$ and we also have $P_i(T)\in \bk_\wstar^\pi(\cM)^H$ because, for all $r,s\in G$ and $k\in H,$ combining the left invariance of the counting measure of $H$ with the fact that $q_{rk}\equiv q_{rkH}=q_{rK}\equiv q_r$ we get 
    \begin{align*}
        P_i(T)_{rk,sk}&=\sum_{h\in H} \pi(p\tau_{q_{rk}}(f_i(q_{rk}^{-1}rk h))^*)T_{rkh,skh}\pi(\tau_{q_{sk}}(f_i(q_{sk}^{-1}s kh))p)\\
        &=\sum_{h\in H} \pi(p\tau_{q_r}(f_i(q_r^{-1}r h))^*)T_{rh,sh}\pi(\tau_{q_s}(f_i(q_{s}^{-1}s h))p)
        = P_i(T)_{r,s}.
    \end{align*}
   
   We claim that if $T\in \bk_\wstar^\pi(\cM)^H$, then $\{P_i(T)\}_{i\in I}$ $\wstar$-converges to $T$.
   Since the net is bounded, it suffices to show that for all $r,s\in G$ and $x,y\in X$, $\lim_i\langle P_i(T)_{r,s}x,y\rangle= \langle T_{r,s}x,y\rangle. $
   
   By construction, $Z_t$ is the center of the W*-ideal of $M_e$ generated by $M_t{M_t}^*$ and $\sigma_t(a)m=ma$ for all $t\in G,$ $m\in M_t$ and $a\in Z_{\tmu}.$   
   Recall that $p$ is the unit of $Z$ and $Z_t=Z\cap \tau_t(Z)$ because $\tau$ is an enveloping action of $\sigma$.
   Hence, $p\tau_t(p)$ is the unit of $Z_t$.
   Denoting $k$ the kernel of $T$, we have $k(rh,sh)=k(r,s)$ for all $h\in H$.
   Using these properties and the fact that $N$ is abelian we get
   \begin{align*}
       P_i(T)_{r,s}
       & = \sum_{h\in H} \pi (p\tau_{q_r}(f_i(q_r^{-1}r h))^*k_T(rh,sh)\tau_{q_s}(f_i(q_{s}^{-1}s h))p)\\
       & =\pi\left( \sum_{h\in H} p\tau_{q_r}(f_i(q_r^{-1}r h))^*k_T(r,s)\tau_{q_s}(f_i(q_{s}^{-1}s h))p\right)\\
       & =\pi\left( \sum_{h\in H} p\tau_{q_r}(f_i(q_r^{-1}r h))^*k_T(r,s) 
 p\tau_{\smu r}(p)
 \tau_{q_s}(f_i(q_{s}^{-1}s h))p\right)\\
 & =\pi\left( \sum_{h\in H} p\tau_{q_r}(f_i(q_r^{-1}r h))^*\sigma_{r\smu}[p\tau_{\smu r}(p)
 \tau_{q_s}(f_i(q_{s}^{-1}s h))]  k_T(r,s) 
 \right)\\
 & =\pi\left( \sum_{h\in H} p\tau_{q_r}(f_i(q_r^{-1}r h))^*\tau_{r\smu}[p\tau_{\smu r}(p)
 \tau_{q_s}(f_i(q_{s}^{-1}s h))]  
 \right)\pi(k_T(r,s))\\
 & =\pi\left( p\tau_{r\smu}(p) \sum_{h\in H} \tau_{q_r}(f_i(q_r^{-1}r h))^*\tau_{r\smu q_s}(f_i(q_{s}^{-1}s h))
 \right)\pi(k_T(r,s))\\
 & =\pi\left( p\tau_{r\smu}(p) \tau_{q_r}\left[\sum_{h\in H} f_i(q_r^{-1}r h)^*\tau_{q_r^{-1} r\smu q_s}(f_i(q_{s}^{-1}s h))
 \right]\right)\pi(k_T(r,s))\\
 & =\pi\left( p\tau_{r\smu}(p) \tau_{q_r}\left[\sum_{h\in H} f_i(h)^*\tau_{q_r^{-1} r\smu q_s}(f_i(q_{s}^{-1}s\rmu q_r h))
 \right]\right)\pi(k_T(r,s)).
   \end{align*}
Notice that $k:=q_r^{-1} r\smu q_s\in H$.
Recalling that $\{f_i\}_{i\in I}$ is a net witnessing the amenability of $\tau$ we get that in the weak operator topology
\begin{align*}
     \lim_i P_i(T)_{r,s}
     &=\lim_i\pi\left( p\tau_{r\smu}(p) \tau_{q_r}\left[\sum_{h\in H} f_i(h)^*\tau_k(f_i(k^{-1}h))
 \right]\right)\pi(k_T(r,s))\\
 &=\pi\left( p\tau_{r\smu}(p)\right)\pi(k_T(r,s))
 = \pi( p\tau_{r\smu}(p)k_T(r,s)) = \pi(k_T(r,s))=T_{r,s}.
\end{align*}

The ccp maps from $\bB(X\otimes \ell^2(G) )$ to $\bB(X\otimes \ell^2(G) )$ form a compact topological space with the topology of pointwise $\wstar$-convergence, so a subnet of $\{P_i\}_{i\in I}$ has a limit $P$.
Since, for all $i\in I$, $P_i(\bk_\wstar(\cM))\sbe \bk_\wstar(\cM)^H$ and both  $\bk_\wstar(\cM)$ and $\bk_\wstar(\cM)^H$ are closed in the $\wstar$-topology, $P(\bk_\wstar(\cM))\sbe \bk_\wstar(\cM)^H$ .
In addition, for all $T\in \bk_\wstar(\cM)^H,$ $T=\lim_i P_i(T) = P(T).$
By Tomiyama's Theorem, $\bk_\wstar(\cM)\to \bk_\wstar(\cM)^H, $ $T\mapsto P(T),$ is a conditional expectation.
\end{proof}

\begin{proposition}\label{prop: ell infty M is amenable}
    The W*-Fell bundle \(\ell^\infty(G, \cM)\) is W$^*$-amenable.
\end{proposition}
\begin{proof}
    The action $\tau$ on the W*-algebra $\ell^\infty(G),$ $\tau_s(f)(t)=f(\smu t),$ is W$^*$-amenable because the function $P\colon \ell^\infty(G,\ell^\infty(G))\to \ell^\infty(G)$, $P(f)(t)=f(t)(t),$ is an equivariant conditional expectation.

    Let $\sigma=(\{Z_t\}_{t\in G},\{\sigma_t\})$ be the central partial action of $\cM$ and $\sigma^e$ its enveloping action, with enveloping algebra $Z^e.$
    The center of $\ell^\infty(G,M_e)$ is $\ell^\infty(G,Z)$ and we claim that the partial action of $G$ on $\ell^\infty(G,Z)$ defined by $\ell^\infty(G,\cM)$ is $\tau\overline{\otimes}\sigma=(\{\ell^\infty(G,Z_t)\}_{t\in G},\{ \tau_t\overline{\otimes}\sigma_t \}_{t\in G})$ with $\tau_t\overline{\otimes}\sigma_t(f)(s)=\sigma_t(f(\tmu s)),$ for all $s,t\in G$ and $f\in \ell^\infty(G,Z_\tmu).$
    To prove this we need to identify the W*-ideal of $\ell^\infty(G,M_e)$ generated by $\ell^\infty(G,M_t)^*\ell^\infty(G,M_t).$
    For all $f,g\in \ell^\infty(G,M_t)$ and $s\in G$, $f^*g(s)=f^*(s)g(ts) = f(ts)^*g(ts).$
    These products W*-generate $\ell^\infty(G, \cspn^\wstar M_t^*M_t )$ and the centre of this algebra is $\ell^\infty(G,Z_\tmu)$ because, by definition, $Z_\tmu=\cspn^\wstar M_t^*M_t.$
    For $f\in \ell^\infty(G,M_t)$, $g\in \ell^\infty(G,Z_\tmu)$ and $s\in G$ we have, $fg (s)=f(s)g(\tmu s)=\sigma_t(g(\tmu s))f(s)=[\tau_t\overline{\otimes}\sigma_t(g)f](s).$
    Then $\tau\overline{\otimes}\sigma$ is the W*-partial action of $G$ on $Z(\ell^\infty(G,M_e))$ defined by $\ell^\infty(G,\cM)$.

    Notice that $\tau\overline{\otimes}\sigma^e$ is the W*-enveloping action of $\tau\overline{\otimes}\sigma$ and $\tau\overline{\otimes}\sigma^e$ is W$^*$-amenable by \cite[Proposition 3.9]{ADactionII1982}.
    Hence, $\ell^\infty(G,\cM)$ is W$^*$-amenable.
\end{proof}

The definition of conditional expectations between Fell bundles, particularly W*-Fell bundles, was introduced by Exel \cite[Definition 21.19]{ExlibroAMS}. We adopt it without requiring $\wstar$-continuity.

The next Theorem below is inspired in \cite{ADaction1979}, specifically  Definition 3.4 and Proposition 4.1 (see Corollary \ref{cor:kernel and cross sectional algebra for Wstar semidirect product bundle} and Example \ref{exa: ell infty of semidirect product bundle}), and complements the characterization of W$^*$-amenability in terms of the W*AP, see \cite[Theorem 6.7]{abadie2021amenability}.

\begin{theorem}\label{thm: new characterization}
    For every W*-Fell bundle $\cM=\{M_t\}_{t\in G},$ the following are equivalent.
    \begin{enumerate}[(a)]
        \item\label{item: M WAD amenable} $\cM$ is W$^*$-amenable.
        \item\label{item: cond exp ell infty} There exists a conditional expectation $P\colon \ell^\infty(G,\cM)\to \cM.$
        \item\label{item: cond exp k} There exists a conditional expectation $Q\colon \bk_\wstar(\cM)\to W^*_\red(\cM).$
    \end{enumerate}
\end{theorem}
\begin{proof}
   Using Proposition \ref{prop: ell infty M is amenable} and \cite[Theorem 6.7 \& Corollary 7.9]{abadie2021amenability} we get \eqref{item: cond exp ell infty} $\Rightarrow$ \eqref{item: M WAD amenable}.
   Combining Theorems \ref{thm: Mh amenble cond exp to fixed point algebra} and \ref{thm: W subalgebra} we obtain \eqref{item: M WAD amenable} $\Rightarrow$ \eqref{item: cond exp k}.
   To complete the proof we assume \eqref{item: cond exp k} and prove \eqref{item: cond exp ell infty} holds.
    To construct, for each $t\in G,$ a map $E_t\colon W^*_\red(\cM)\to M_t$ like that of \cite[Lemma 17.8]{ExlibroAMS} it is convenient to identify $W^*_\red(\cM)\equiv W^*_{\red\pi}(\cM)$ and to regard each fiber $M_t$ as a subset of $\bB(X)$ by identifying $m=\pi(m)$.
    We define $E_t\colon W^*_\red(\cM)\to M_t$ by $E_t(T) = T_{t,e}$.
    Note that $E_t$ is a linear contraction and that $E_t(T)^* =(T_{t,e})^* = {T^*}_{e,t}= {T^*}_{\tmu ,e} = E_\tmu(T^*).$
    Hence, $E_t(T)^* = E_\tmu(T^*).$
    
    Let $Q\colon \bk_\wstar(\cM)\to W^*_\red(\cM)$ be a conditional expectation which, using Theorem \ref{thm: iso algebra of kernels}, we view as a conditional expectation $Q\colon W^*_\red(\ell^\infty(G,\cM))\to W^*_\red(\cM).$
    Define $P\colon \ell^\infty(G,\cM)\to \cM$ as the map sending $f\in \ell^\infty(G,M_t)$ to $E_t(Q(f)).$
    In other words, $P$ is the disjoint union of the family $\{E_t\circ Q\colon \ell^\infty(G,M_t)\to M_t\}_{t\in G}.$
    All we need to do now is to check $P$ satisfies the conditions stated in \cite[Definition 21.19]{ExlibroAMS}.

    Clearly, $E_t\circ Q$ is an idempotent linear contraction for all $t\in G.$
    In addition, $E_e\circ Q\colon \ell^\infty(G,M_e)\to M_e$ is a conditional expectation in the usual sense.
    For every $b\in \ell^\infty(G,M_t)$ we have $(E_t\circ Q (b))^* =E_\tmu(Q(b)^*) =E_\tmu (Q(b^*))=E_\tmu\circ Q(b^*).$
    If $a\in M_s,$ then 
    \begin{align*}
        E_{ts}\circ Q (ba) &= (Q(ba))_{ts,e}=(Q(b)a)_{ts,e}
        = \sum_{r\in G} Q(b)_{ts,r}(\lambda_s\otimes a)_{r,e}
        = Q(b)_{ts,s} a\\
        & = Q(b)_{t,e}a= E_t\circ Q(b)a
    \end{align*}
    and it follows that $E_{ts}\circ Q (ba) = E_t\circ Q(b)a.$
    Consequently, 
    \begin{equation*}
        E_{st}\circ Q(ab) = (E_{\tmu \smu}\circ Q(b^*a^*))^* = (E_\tmu\circ Q(b^*) a^*)^* = aE_t\circ Q(b),
    \end{equation*}
    which completes the verification of \cite[Definition 21.19]{ExlibroAMS}.
\end{proof}

Imitating \cite[Corollaire 4.2]{ADaction1979} we obtain

\begin{corollary}[part of {\cite[Proposition 7.1]{abadie2021amenability}}]\label{cor: injectivity and amenability}
    The following are equivalent:
    \begin{enumerate}
        \item $M_e$ in injective and $\cM$ is W$^*$-amenable.
        \item $W^*_\red(\cM)$ is injective.
    \end{enumerate}
\end{corollary}
\begin{proof}
    Assume the first condition holds. 
    As noticed in the proof of \cite[Proposition 7.1]{abadie2021amenability}, $\bk_\wstar(\cM)$ is injective and Theorem \ref{thm: new characterization} gives a conditional expectation $Q\colon \bk_\wstar(\cM)\to W^*_\red(\cM),$ then $W^*_\red(\cM)$ is injective.
    Conversely, if $W^*_\red(\cM)$ is injective, then $M_e$ is injective because there exists a canonical conditional expectation $W^*_\red(\cM)\to M_e.$
    Using $\pi\colon \cM\to \bB(X)$ we may think $W^*_\red(\cM)\equiv W^*_{\red\pi}(\cM)\sbe \bk_\wstar^\pi(\cM)\equiv \bk_\wstar(\cM)\sbe \bB(\ell^2(G,X)).$
    Since $W^*_\red(\cM)$ is injective, there exists a conditional expectation $\bB(X\otimes\ell^2(G))\to W^*_\red(\cM)$ which we can restrict to get a conditional expectation $\bk_\wstar(\cM)\to  W^*_\red(\cM).$
    Then Theorem \ref{thm: new characterization} implies $\cM$ is W$^*$-amenable.
\end{proof}

\begin{theorem}\label{thm: M amena iff N and Mh amenable}
    If $H$ is a normal subgroup of $G$, then the following are equivalent
    \begin{enumerate}
        \item $\cM$ is W$^*$-amenable.
        \item $\cM_H$ and $\dot{\cM}$ are W$^*$-amenable.
    \end{enumerate}
\end{theorem}
\begin{proof}
    Assume $\cM$ is amenable.
    In \cite[pp 199]{abadie2021amenability} it is shown that $\cM_H$ is W$^*$-amenable.
    The argumentation goes as follows.
    Let $Z,\ \tilde{Z},\ N,\ \sigma,\ \tilde{\tau}$ and $\tau$ be as in the proof of Theorem \ref{thm: Mh amenble cond exp to fixed point algebra}.
    By \cite[Definition 5.16]{abadie2021amenability}, $\tau$ is amenable and \cite[Proposition 2.6]{ADactionII1982} implies $\tau|_H$ is amenable.
    Using the unit $\tilde{p}$ of $\tilde{Z}$ we can construct an $H$-equivariant conditional expectation $P\colon N\to \tilde{Z},$ $P(n)=\tilde{p}n;$ and we can use \cite[Proposition 3.8]{ADaction1979} to deduce that $\tilde{\tau}$ is amenable. Since this last action is the W*-enveloping action of the W*-partial action of $H$ on $Z(M_e)$ defined by $\cM_H$, $\cM_H$ is W$^*$-amenable (by definition).

   Theorem \ref{thm: new characterization} gives a conditional expectation $Q\colon \bk_\wstar(\cM)\to W^*_\red(\cM)$ that we can restric to get a conditional expectation $Q^H\colon \bk_\wstar(\cM)^H\to  W^*_\red(\cM).$
   Conjugating $Q^H$ using the isomorphism of Proposition \ref{thm: identificacion of kernels of WMH} we get a conditional expectation $\bk_\wstar(\dot{\cM})\to W^*_\red(\dot{\cM})$.
   Thus, Theorem \ref{thm: new characterization} implies $\dot{\cM}$ is W$^*$-amenable.

   To prove the converse we combine Theorem \ref{thm: Mh amenble cond exp to fixed point algebra}, Proposition \ref{thm: identificacion of kernels of WMH} and Theorem \ref{thm: new characterization} to get conditional expectations $P\colon  \bk_\wstar(\cM)\to \bk_\wstar(\cM)^H$ and $Q\colon \bk_\wstar(\cM)^H\to W^*_\red(\cM)$ that give a conditional expectation $Q\circ P\colon  \bk_\wstar(\cM)\to W^*_\red(\cM)$.
\end{proof}

\section{Fell bundles and group coactions}\label{sec:coactions}

In this section, we reinterpret our results in the framework of group coactions by examining the relationship between Fell bundles and the corresponding group coactions. While this connection is well established in the setting of \cstar {}algebras, our aim here is to extend these ideas to the context of von Neumann algebras.

Let $G$ be a discrete group.\footnote{The theory of group coactions extends naturally to locally compact groups, see \cite{echterhoff2006categorical}. Since we only use discrete groups in this paper, we restrict to this case to avoid complications.} Recall that a coaction of $G$ on a \cstar {}algebra $A$ is a $*$-homomorphism
$$\delta\colon A\to A\otimes C^*(G)$$
satisfying $(\delta\otimes\id)\circ\delta=(\id\otimes\delta_G)\circ\delta$, where $\delta_G\colon C^*(G)\to C^*(G)\otimes C^*(G)$ is the usual comultiplication: $\delta_G(t)=t\otimes t$ for $t\in G$. Moreover, as usual, we assume that $\delta$ is \emph{nondegenerate} in the sense that $\cspn\, \delta(A)(1\otimes C^*(G))=A\otimes C^*(G)$. This is equivalent to $(\epsilon\otimes\id)\circ\delta=\id_A$, where $\epsilon \colon C^*(G)\to \C$ is the trivial representation: $\epsilon(t)=1$ for all $t\in G$. In particular, every coaction is automatically an injective homomorphism.

The coaction $\delta$ is called \emph{normal} if $\delta^\lambda:=(\id\otimes\lambda)\circ\delta\colon A\to A\otimes C^*_\red(G)$ is injective. In this case $\delta^\lambda$ may be viewed as a ``reduced coaction'' of $G$, see \cite{echterhoff2006categorical} and references there in.

Given a coaction $\delta\colon A\to A\otimes C^*(G)$, its \emph{crossed product} is defined as
$$A\rtimes_\delta\dual G:=\cspn\, \delta^\lambda(A)(1\otimes M(C_0(G)))\sbe \bB(A\otimes \K(\ell^2(G)))=\bB(A\otimes\ell^2(G)),$$
where $M\colon C_0(G)\to \bB(\ell^2(G))$ denotes the representation by multiplication operators: $M_f(\xi)=f\cdot \xi$.

The crossed product $A\rtimes_\delta\dual G$ carries a canonical $G$-action, called the dual action, and denoted by $\dual\delta$; it is given on generators by $\dual\delta_t(\delta^\lambda(a)(1\otimes M_f))=\delta^\lambda(a)(1\otimes M_{\rt_t(f)})$, where $\rt_t(f)(s):=f(st)$ denotes the right translation $G$-action on $C_0(G)$.

The coaction $\delta$ is called \emph{maximal} if the canonical homomorphism
$$\Phi_\delta\colon A\rtimes_\delta\dual G\rtimes_{\dual\delta}G\onto A\otimes\K(\ell^2(G))$$
is injective (and hence an isomorphism). Equivalently, $\delta$ is maximal if and only if it lifts to a (necessarily) injective homomorphism $\delta_{\max}\colon A\to A\otimes_{\max} C^*(G)$, see 
\cite[{Theorem~5.1}]{BssEffMaximality}. 

Given a coaction $\delta\colon A\to A\otimes C^*(G)$, its spectral subspaces
$$A_t:=\{a\in A: \delta(a)=a\otimes t\}$$
form a Fell bundle $\A=(A_t)_{t\in G}$ concretely represented in $A$.
We call $\A$ the \emph{spectral Fell bundle} of $\delta$.
Moreover, the algebraic direct sum $\oplus_{t\in G}^\mathrm{alg} A_t$ is a (norm) dense $*$-subalgebra of $A$, canonically isomorphic to the cross-sectional $*$-algebra $C_c(\A)$ consisting of finitely supported sections $G\to \A$.

Hence each coaction $\delta$ of $G$ on $A$ yields a $G$-grading for the \cstar {}algebra $A$. This is a \emph{topological grading} in the sense that we have projections $E_t\colon A\to A_t$; these are given by $E_t(a):=(\id\otimes\varphi_t)(\delta(a))$, where $\varphi_t\colon C^*(G)\to \C$ is the bounded linear functional specified by $\varphi_t(s)=\delta_{s,t}$ for $s,t\in G$. Notice that $\varphi_e$ is just the canonical trace on $C^*(G)$. The functional $\varphi_t$ is given by $\varphi_t(x)=\varphi_e(t^{-1}x)$. The functional $\varphi_t$ factors through $C^*_\red(G)$, where it becomes the vector functional $x\mapsto \braket{\delta_t}{x(\delta_e)}$.

The following result is well known:

\begin{theorem}\label{theo:equivalence-Fell-bundles-coactions}
Given a Fell bundle $\A=(A_t)_{t\in G}$ over $G$, its full cross-sectional \cstar {}algebra $C^*(\A)$ carries a maximal coaction $\delta_\A\colon C^*(\A)\to C^*(\A)\otimes C^*(G)$ satisfying $\delta_\A(a_t)=a_t\otimes t$ for $a_t\in A_t$. The spectral Fell bundle of $\delta_\A$ is canonically isomorphic to $\A$. Every maximal coaction is isomorphic to one of this form for an essentially unique Fell bundle, namely, the spectral Fell bundle of the coaction.

The reduced cross-sectional \cstar {}algebra $C^*_\red(\A)$ carries a normal coaction $\delta_\A^\red\colon C^*_\red(\A)\to C^*_\red(\A)\otimes C^*(G)$ satisfying $\delta_\A^\red(a_t)=a_t\otimes t$ for $a_t\in A_t$. Every normal coaction is isomorphic to one of this form via the associated spectral Fell bundle.
\end{theorem}

The above result yields an equivalence between the categories of Fell bundles over $G$ (with usual morphisms of Fell bundles over $G$) and the category of maximal (resp. normal) coactions of $G$. A morphism between two $G$-coactions $(A,\delta)$ and $(A',\delta')$ is a just a $*$-homomorphism $\pi\colon A\to A'$ that is equivariant with respect to the coactions, meaning that $\delta'(\pi(a))=(\pi\otimes\id)(\delta(a))$ for all $a\in A$.

\subsection{Coactions on von Neumann algebras and W*-Fell bundles}

Next we consider coactions on von Neumann algebras: recall that a coaction of $G$ on a von Neumann algebra $N$ is an injective, normal and unital $*$-homomorphism $\delta\colon M\to M\bar\otimes W^*_\red(G)$ satisfying the coaction identity $(\delta\otimes\id)\circ \delta=(\id\otimes\Delta_G)\circ\delta$, where here $\Delta_G$ denotes the comultiplication on the group von Neumann algebra $W^*_\red(G)=C^*_\red(G)''\sbe \bB(\ell^2(G))$. Thus coactions on von Neumann algebras are analogues of reduced \cstar {}coactions. 

\begin{remark}\label{rem:coactions-vN-injective-ndeg}
For a coaction $\delta\colon M\to M\bar\otimes W^*_\red(G)$ as above, the linear span of elements of the form $\delta(m)(1\otimes\lambda_t)$ with $m\in M$ and $t\in G$ is weak*-dense in $M\bar\otimes W^*_\red(G)$. For coactions on \cstar {}algebras this is proved in \cite[Corollaire~7.15]{Baaj-Skandalis}; a similar proof applies to coactions on von Neumann algebras.
\end{remark}

The following is probably well known, but we could not find a reference, so we add a proof.

\begin{proposition}
Given a \Wstar coaction $\delta\colon M\to M\bar\otimes W^*_\red(G)$ as above, its spectral Fell bundle is a \Wstar Fell bundle $\M=(M_t)_{t\in G}$ that yields a $G$-grading for the von Neumann algebra $M$, in the sense the algebraic direct sum $\oplus_{t\in G}^{\mathrm{alg}}M_t$ is weak*-dense in $M$.
\end{proposition}
\begin{proof}
Since $\delta$ is normal, it is clear that $M_t$ is a weak*-closed subspace of $M$, so that $\M=(M_t)_{t\in G}$ is a \Wstar Fell bundle.
To see that $\oplus_{t\in G}^{\mathrm{alg}}M_t$ is weak*-dense in $M$, it is enough to prove that if $\omega\in M_*$ is such that $\omega(M_t)=0$ for all $t\in G$, then $\omega=0$. Fix $\omega$ with this property. Then $\omega(E_t(m))=0$ for all $m\in M$ and $t\in G$. Recall that $E_t(m)=(\id\otimes\varphi_t)(\delta(m))$, where $\varphi_t\in W^*_\red(G)_*=A(G)$ (the Fourier algebra) is the vector functional given by $\varphi_t(x)=\braket{\delta_t}{x(\delta_e)}$ for $x\in W^*_\red(G)$. These functionals separate the points: if $\varphi_t(x)=0$ for all $t\in G$, then $x=0$. But then, if $\omega(E_t(m))=0$ for all $t\in G$ and $m\in M$, we get $\varphi_t((\omega\otimes\id)(\delta(m)))=0$ for all $t\in G$, so that $(\omega\otimes\id)(\delta(m))=0$ for all $m\in M$. It follows that $(\omega\otimes\id)(\delta(m)(1\otimes\lambda_t))=0$ for all $m\in M$ and $t\in G$ and Remark~\ref{rem:coactions-vN-injective-ndeg} then implies that $\omega=0$, as desired. \end{proof}

The following result is the \Wstar analogue of Theorem~\ref{theo:equivalence-Fell-bundles-coactions} and yields an equivalence between the categories of \Wstar Fell bundles over $G$ and \Wstar coactions of $G$. 

\begin{theorem}\label{theo:spectral-Fell-bundle}
    Given a \Wstar Fell bundle $\M=(M_t)_{t\in G}$, its cross-sectional \Wstar algebra $W^*_\red(\M)$ carries a coaction $$\delta_\M\colon W^*_\red(\M)\to W^*_\red(\M)\bar\otimes W^*_\red(G)$$ satisfying $\delta_\M(m_t)=m_t\otimes\lambda_t$ for all $t\in G$ and $m_t\in M_t$.
    The spectral \Wstar Fell bundle of $\delta_\M$ is canonically isomorphic to $\M$. Moreover, every \Wstar coaction of $G$ is of this form for a unique \Wstar Fell bundle, up to natural isomorphism, namely its spectral Fell bundle.
\end{theorem}
\begin{proof}
    The construction of the coaction $\delta_\M$ is already given in \cite[Section 5.5]{abadie2021amenability}. The basic idea is the same as the case of (reduced) discrete group coactions. We provide some of the main details. The coaction $\delta_\M$ has a canonical ``unitary implementation'' as follows: consider the Hilbert $M_e$-module $\ell^2_\wstar(\cM\times G)\cong \ell^2_\wstar(\cM)\otimes\ell^2(G)$, where $\M\times G$ denotes the \Wstar Fell bundle over $G\times G$ obtained as the pullback of $\M$ along the first coordinate projection $G\times G\to G$. We define a unitary $U$ on $\ell^2_\wstar(\M\times G)$ by the formula
    $$U(\zeta)(s,t):=\zeta(s,s^{-1}t)$$
    for a $\ell^2$-section $\zeta$ of $\M\times G$ and $s,t\in G$. It is easy to see that $W$ is, indeed, a unitary with $U^*(\zeta)(s,t)=\zeta(s,st)$.
    Moreover, using the canonical (normal) representation of $W^*_\red(\M)\bar\otimes W^*_\red(G)$ as operator on $\ell^2_\wstar(\cM\times G)$, a straightforward computation shows that $U(m_t\otimes 1)U^*=m_t\otimes \lambda_t$, so that the normal homomorphism $m\mapsto U(m\otimes 1)U^*$ yields the desired coaction $\delta_\M$ as in the statement. The spectral Fell bundle decomposition of $\delta_\M$ is canonically isomorphic to the original Fell bundle $\M$ because, by construction, $M_t$ is canonically embedded into the $t$-spectral subspace $W^*_\red(\M)_t$, and the projection $E_t$ gives the inverse map.

    Conversely, starting with a \Wstar coaction $\delta\colon M\to M\bar\otimes W^*_\red(G)$, its spectral Fell bundle $\M=(M_t)_{t\in G}$ is a \Wstar Fell bundle. Notice that $E:=E_e=(\id\otimes\varphi_e)\circ \delta \colon M\onto M_e$ is a faithful normal conditional expectation. From $E$ we obtain a \Wstar Hilbert $M_e$-module $\ell^2_E(M)$, defined as the self-dual completion of $M$ endowed with the $M_e$-valued inner product $\braket{m}{n}_E:=E(m^*n)$. And since $E$ is faithful, so is the canonical representation $\pi\colon M\to \bB(\ell^2_E(M))$. Now we notice that the canonical algebraic isomorphism $\oplus_{t\in G}^{\mathrm{alg}}M_t\to C_c(\M)$ extends to an isomorphism of \Wstar Hilbert $M_e$-modules $\ell^2_E(M)\cong \ell^2_\wstar(\M)$; this follows from the computation 
    $$\braket{m_s}{m_t}_E=E(m_s^*m_t)=\braket{m_s}{m_t}_{M_e}$$
    for $m_s\in M_s$ and $m_t\in Mt$. Indeed, except if $s=t$, the above inner products vanish, and if $s=t$, they equal $m_s^*m_t$ as $E$ is just the identity map on $M_e$. The representation $\pi$ then yields a unital normal injective $*$-homomorphism from $W^*_\red(\M)$ into a  $\bB(\ell^2_\wstar(\M)\otimes \ell^2(G))\cong \bB(\ell^2_E(M)\otimes\ell^2(G))$ sending $m_t\in M_t$ to $m_t\otimes \lambda_t$. This is just a manifestation of the coaction $\delta_\M$. On the other hand, $\delta$ identifies $M$ with its image in $M\otimes W^*_\red(G)\sbe \bB(\ell^2_E(\M)\otimes\ell^2(G))$. This implies that the algebraic isomorphism $\oplus_{t\in G}^{\mathrm{alg}}M_t\to C_c(\M)$ extends to an isomorphism of \Wstar algebras $M\cong W^*_\red(\M)$.
\end{proof}

Recall that the \Wstar{}crossed product of a coaction $\delta\colon M\to M\bar\otimes W^*_\red(G)$ is defined as 
$$M\rtimes_\delta G:=\cspn^\wstar\{\delta(a)(1\otimes f): a\in M, f\in C_0(G)\}\sbe M\bar\otimes\bB(\ell^2(G)),$$
where we view $M\bar\otimes W^*_\red(G)$ as a \cstar {}subalgebra of $M\bar\otimes \bB(\ell^2(G))$ and also embed $C_0(G)\into \bB(\ell^2(G))$ via the representation by multiplication operators. 

Let us also recall Proposition~5.28 and~Corollary 5.29 from \cite{abadie2021amenability}:

\begin{proposition}\label{pro:equivariant-iso-dual-kernel}
    Given a \Wstar{}Fell bundle $\M$ over $G$, there exists a canonical $G$-equivariant isomorphism
    $$\bk_{\wstar}(\M)\cong W^*_\red(\M)\bar\rtimes_{\delta_\M} G$$
    that identifies an element of the form $\delta_\M(a)(1\otimes f)\in W^*_\red(\M)\bar\rtimes_{\delta_\M} G$ with the kernel $k_{a,f}(s,t):=a(st^{-1})f(t)$ in $\bk_{\wstar}(\M)$. Taking \Wstar{}crossed products, this decends to an isomorphism:
    $$\bk_{\wstar}(\M)\bar\rtimes_{\beta^{\wstar}}G\cong W^*_\red(\M)\bar\otimes\bB(\ell^2(G)).$$
\end{proposition}

Next we allow for normal subgroups $H\trianglelefteq G$ and discuss restricted coactions from $G$ to $\dot G $. In the general context of quantum groups, one can restrict every coaction to a quantum subgroup, see \cite{Vaes:New-approach}.\footnote{We would like to thank Stefaan Vaes for pointing out this to us.} For W*-coaction $\delta\colon M\to M\otimes W^*_\red(G)$, this corresponds to restricting $\delta$ from $G$ to a coaction 
$$\dot\delta\colon M\to M\bar\otimes W^*_\red(\dot G)$$ of the quotient $\dot G:=G/H $ by a normal subgroup $H\trianglelefteq G$. More precisely, this is the unique W*-coaction of $\dot G$ satisfying
$$(\dot\delta\otimes\id)\circ\delta =(\id\otimes\dot\delta_G)\circ\delta,$$
where $\dot\delta_G\colon W^*_\red(G)\to W^*_\red(G)\bar\otimes W^*_\red(\dot G)$ is the ``restriction'' of the comultiplication $\delta_G\colon W^*_\red(G)\to W^*_\red(G)\bar\otimes W^*_\red(G)$. It satisfies
$$\dot\delta_G(\lambda_t)=\lambda_t\otimes\dot\lambda_{\dot t},$$
where $\dot\lambda\colon \dot G\to \bB(\ell^2(\dot G))$ denotes the left regular representation of $\dot G$, and $\dot t:=tH$ denotes the class of $t\in G$ in $\dot G$. The restricted coaction $\dot\delta$ then necessarily satisfies
$$\dot\delta(m)=m\otimes\dot\lambda_{\dot t}\quad \mbox{for all }t\in G, m\in M_t.$$
    Here is how the restricted coaction $\dot\delta$ can be obtained in a more explicit way using a unitary implementation: we consider the spectral Fell bundle $\M$ of $\delta$, and look at the Fell bundle $\M\times \dot G$ over $G\times \dot G$ defined as the pull back of $\M$ along the first coordinate projection $G\times \dot G\onto G$. Then, as in the construction of the coaction $\delta\cong\delta_\M$, we can construct a unitary operator $\dot U$ on $\ell^2_{\wstar}(\M\times \dot G)\cong \ell^2_{\wstar}(\M)\otimes\ell^2(\dot G)$ by the formula
    \begin{equation}\label{eq:def-dot-U}
    \dot U(\zeta)(s,\dot t):=\zeta(s,s^{-1}\dot t).
    \end{equation}
    Of course, here we use the canonical left translation action of $G$ on $\dot G=\dot G $. Straightforward computations show that $\dot U(m\otimes 1)\dot U^*=m\otimes \dot\lambda_{\dot t}$ if $m\in M_t$ for $t\in G$. It follows that the map $\dot\delta\colon M\to M\bar\otimes W^*_\red(\dot G)\sbe \bB(\ell^2_{\wstar}(\M\times \dot G))$ defined by 
    $$\dot\delta(m):=\dot U(x\otimes 1)\dot U^*$$
    yields the desired restricted W*-coaction $\dot\delta\colon M\to M\bar\otimes W^*_\red(\dot G)$.

We now describe the restriction of coactions in terms of their spectral decompositions.

\begin{proposition}
    Every coaction $\delta\colon M\to M\bar\otimes W^*_\red(G)$ ``restricts'' to a coaction $\dot\delta\colon M\to M\bar\otimes W^*_\red(\dot G)$, in such a way that $a\in M_t=\{a\in M:\delta(a)=a\otimes\lambda_t\}$ is sent to $\dot\delta(a)=a\otimes \dot\lambda_{\dot t}$ for every $t\in G$. Moreover, if $\M$ denotes the spectral Fell bundle of $\delta$, then the spectral Fell bundle of $\dot\delta$ is canonically isomorphic to the \Wstar{}partial cross-sectional Fell bundle $\M/H$ of $\M$ (as in Definition~\ref{Def:Partial-Cross-Sec}).
    \end{proposition}
    \begin{proof}
    Let $\M=(M_t)_{t\in G}$ (resp. $\dot\M=(\dot M_{\dot t})_{\dot t\in \dot G}$) be the spectral Fell bundle of $\delta$ (resp. $\dot\delta$). Identifying $M\cong W^*_\red(\M)$ as in Theorem~\ref{theo:spectral-Fell-bundle}, and using the construction of $\M/H$ (see discussion before Definition~\ref{Def:Partial-Cross-Sec}), the fiber of $\M/H$ at $\dot t\in \dot G$ may be identified with the subspace
    $$W^*_\red(\M_u)=\cspn^\wstar\{\delta(m)=m\otimes s: s\in u, m\in M_s\}\sbe M\bar\otimes \bB(\ell^2(G)).$$
    Since $\delta$ is an isomorphism onto its image, we may further identify this with the $\wstar$-closed subspace of $M$ generated by $M_s$ with $s\in G$ in the same class of $t$, i.e., $\dot s=\dot t$. 
    All we need to show is then: 
    $$\dot M_{\dot t}= \cspn^\wstar \{m\in M_s: s\in G, \dot s=\dot t\}.$$
    The inclusion $``\supseteq"$ is clear because 
    the coaction $\dot\delta$ satisfies $\dot\delta(m)=m\otimes\dot\lambda_{\dot t}$ for all $t\in G$.
    To show that inclusion $``\sbe"$, we use that $\dot M_{\dot t}=\dot E_{\dot t}(M)$, where $\dot E\colon M\onto \dot M_{\dot t}$ denotes the spectral projection. But since the subspaces $M_s$ for $s\in G$ generate a W*-dense subspace of $M$, and since $\dot E$ is W*-continuous, it is enough to see that observe that if $m\in M_s$, then 
    $$\dot E_{\dot t}(m)=(\id\otimes \dot\varphi_{\dot t})(\dot\delta(m))=(\id\otimes \dot\varphi_{\dot t})(m\otimes\dot\lambda_{\dot s})=m\in M_s$$
    if $\dot s=\dot t$, and it is zero otherwise.
    \end{proof}

The above result gives an alternative proof for part of the statement in Theorem~\ref{thm: identificacion of kernels of WMH}. In order to provide an alternative proof of the complete statement in the language of coactions, we want to prove:

\begin{theorem}
    Let $\delta\colon M\to M\bar\otimes W^*_\red(G)$ be a W*-coaction and let $H\trianglelefteq G$ be a normal subgroup with corresponding quotient $\dot G=G/H$.
    Consider the W*-crossed product $M\bar\rtimes_\delta G$ endowed with the dual W*-action $\dual\delta\colon G\car M\bar\rtimes_\delta G$. Then the $H$-fixed point algebra (i.e. the fixed point algebra for the restricted action of $H$) is canonically isomorphic to
    $$(M\bar\rtimes_\delta G)^H\cong M\bar\rtimes_{\dot\delta} \dot G,$$
    the W*-crossed product for the restricted coaction $\dot\delta\colon M\to M\bar\otimes W^*_\red(\dot G)$.
\end{theorem}
\begin{proof}
By definition, the W*-crossed product $M\bar\rtimes_{\dot\delta}\dot G$ is the span $\wstar$-closure of elements of the form 
$$\dot\delta(m)(1\otimes f)\in M\bar\otimes\bB(\ell^2(\dot G))$$ with $m\in M$ and $f\in \ell^\infty(\dot G)$. Since $M$ is generated by the spectral subspaces $M_t$ and $\dot\delta(m)=m\otimes\dot\lambda_{\dot t}$ for $m\in M_t$, we can also view $M\rtimes_{\dot\delta}\dot G$ as the span $\wstar$-closure of elements of the form $m\otimes \dot\lambda_{\dot t}f$ with $m\in M_t$, $t\in G$, and $f\in \ell^\infty(\dot G)$. 
Applying $\delta\otimes\id$, we can further identify 
$$M\rtimes_{\dot\delta}\dot G\cong \cspn^\wstar\{m\otimes\lambda_t\otimes\dot\lambda_{\dot t}:m\in M_t, t\in G, f\in \ell^\infty(\dot G)\}\sbe M\bar\otimes\bB(\ell^2(G))\bar\otimes\bB(\ell^2(\dot G)).$$
Let us now consider the unitary $\dot U\in \bB(\ell^2(G\times \dot G))$ given by the same formula as in~\eqref{eq:def-dot-U}. A simple computation shows that 
$$(1\otimes f)U=U\dot\Delta(f)\quad\forall f\in \ell^\infty(\dot G),$$
where $\dot\Delta\colon \ell^\infty(\dot G)\to \ell^\infty(G\times\dot G)$ is defined by $\dot\Delta(f)(s,\dot t):=f(st)$. Here we view $\ell^\infty(\dot G)$ as functions on $G$ that are constant on lateral classes $tH$. Morever, we also represent $\ell^\infty(G\times\dot G)$ as multiplication operators on $\ell^2(G\times\dot G)$. Applying $\Ad_{1\otimes U^*}$, we the get an isomorphism
$$M\rtimes_{\dot\delta}\dot G\cong \cspn^\wstar\{m\otimes (\lambda_t\otimes 1)\dot\Delta(f): m\in M_t, t\in G,f\in \ell^\infty(\dot G)\}\sbe M\bar\otimes\bB(\ell^2(G))\bar\otimes\bB(\ell^2(\dot G)).$$
Notice that
$$\cspn^\wstar\{(\lambda_t\otimes 1)\dot\Delta(f): t\in G, f\in \ell^\infty(\dot G)\}$$
can be viewed as the \Wstar{}crossed product $G\ltimes_{\dot\tau}\ell^\infty(\dot G)$ for the canonical translation $G$-action $\dot\tau$ on $\ell^\infty(\dot G)$. This can be viewed canonically as the \Wstar{}subalgebra of operators on $\ell^2(G\times G)$ generated by the elements of the form $(\lambda_t\otimes)\Delta(f)$ with $t\in G$ and $f\in \ell^2(\dot G)$, where $\Delta(f)(s,t):=f(st)$ is viewed as a multiplication operator now on $\ell^2(G\times G)=\ell^2(G)\otimes\ell^2(G)$. It follows that
$$M\rtimes_{\dot\delta}\dot G\cong \cspn^\wstar\{m\otimes (\lambda_t\otimes 1)\Delta(f): m\in M_t, t\in G,f\in \ell^\infty(\dot G)\}\sbe M\bar\otimes\bB(\ell^2(G))\bar\otimes\bB(\ell^2(G)).$$
On the other hand, using similar arguments, the dual $G$-action, and hence also its restriction to $H$, on the \Wstar{}crossed product $M\rtimes_\delta G$ is \emph{integrable} (see \cite[Proposition~2.5]{Vaes:Unitary}). This implies that the fixed point algebra $(M\rtimes_\delta G)^H$ equals
$$(M\rtimes_\delta G)^H=\cspn^\wstar\{\delta(m)(1\otimes f): m\in M, f\in \ell^\infty(\dot G)\}.$$
This is because the elements of the form $\delta(m)(1\otimes f)$ with $f\in C_c(G)$
form a $\wstar$-dense subset of integrable elements with integral (along the action) given by $\delta(m)(1\otimes E(f))$, where $E(f)=\sum \tau_t(f)$, where $\tau$ denotes the (right) translation $G$-action. But the $\wstar$-closed linear span of those elements equals $\ell^\infty(G)^H=\ell^\infty(\dot G)$. Similar arguments as before (using the map $\delta\otimes\id$ and conjugation with the unitary $1\otimes U^*$, where $U\zeta(s,t)=\zeta(s,s^{-1}t)$) imply that
$$(M\rtimes_\delta G)^H\cong \cspn^\wstar\{m\otimes (\lambda_t\otimes 1)\Delta(f): m\in M_t, t\in G, f\in \ell^\infty(\dot G)\}\cong M\rtimes_{\dot\delta}\dot G,$$
finishing the proof of the theorem.
\end{proof}

\subsection{Aproximation property in terms of coactions}

Since the category of coactions of $G$ is equivalent to that of Fell bundles over $G$, it is natural to ask how the approximation property of Fell bundles translates into the language of coactions.

\begin{definition}
    We say that a \Wstar{}coaction $\delta\colon M\to M\bar\otimes W^*_\red(G)$ is \emph{co-amenable} if there exists a norm-one projection (automatically a conditional expectation) $M\rtimes_\delta G\onto M$.
\end{definition}

\begin{example}
Start with an action $\sigma$ of $G$ on a von Neumann algebra $N$, and consider its dual coaction $\delta=\dual\sigma$ of $G$ on $M=N\bar\rtimes_\sigma G$. Then $\delta$ is co-amenable if and only if $\sigma$ is amenable in the sense of Anantharaman-Delaroche. Indeed, by W*-crossed product duality for \Wstar{}actions, we have $$M\bar\rtimes_\delta G\cong N\bar\rtimes_\sigma G\bar\rtimes_{\dual\sigma}G\cong N\bar\otimes\bB(\ell^2(G)).$$ 
The result then follows from \cite[Proposition~4.1]{ADaction1979}. Alternatively, this will also follow from our next result.    
\end{example}

\begin{theorem}
        A \Wstar{}coaction $\delta\colon M\to M\bar\otimes W^*_\red(G)$ is co-amenable if and only if the dual action $\dual\delta$ of $G$ on $M\bar\rtimes_\delta G$ is W$^*$-amenable, if and only if the spectral \Wstar{}Fell bundle of $\delta$ has the W*AP.
\end{theorem}
\begin{proof}
    This is a consequence of our previous results: by Theorem~\ref{theo:spectral-Fell-bundle} and Proposition~\ref{pro:equivariant-iso-dual-kernel}, we have an equivariant isomorphism $M\bar\rtimes_\delta G\cong \bk_\wstar(\M)$, where $\M$ denotes the spectral Fell bundle of $\delta$. And by Theorem~\ref{thm: new characterization}, the (dual) $G$-action on $\bk_\wstar(\M)$ is amenable if and only if there exists a norm-one projection (a conditional expectation)
    $$M\bar\rtimes_\delta G\cong\bk_\wstar(\M)\onto W^*_\red(\M)\cong M.$$
    By \cite[Theorem~6.8]{abadie2021amenability}, $\M$ has the W*AP if and only if the dual action $\dual\delta$ on $\bk_\wstar(\M)$ is amenable.
\end{proof}

\begin{remark}
  One can go further and consider the notion of co-amenability for \cstar {}coactions—that is, coactions of a group $G$ on a \cstar {}algebra $A$. The most direct approach is to say that a coaction $\delta\colon A \to A \otimes C^*(G)$ is \emph{co-amenable} (in the sense of \cstar {}coactions) if its dual action $\dual\delta$ of $G$ on the crossed product $A \rtimes_\delta G$ is amenable. If $\A$ denotes the spectral Fell bundle associated to $\delta$, then $A \rtimes_\delta G$ identifies with the \cstar {}algebra of kernels $\bk(\A)$, and the dual action corresponds to the canonical (right) translation action of $G$ on $\bk(\A)$. It follows that $\delta$ is co-amenable if and only if the Fell bundle $\A$ has the approximation property.

  An alternative approach to defining co-amenability for \cstar {}coactions is to associate a \Wstar{}coaction to a given coaction $\delta\colon A \to A \otimes C^*(G)$. Although we do not pursue this direction in detail here, the basic idea would be as follows. Instead of taking the bidual of $A$ directly, one first considers the crossed product $A \rtimes_\delta G$, and then its bidual \Wstar{}algebra $(A \rtimes_\delta G)''$. Within this, one identifies the von Neumann algebra $M$ generated by the image of $A$ in $\M(A \rtimes_\delta G) \sbe (A \rtimes_\delta G)''$. This construction should yield a \Wstar{}coaction of $G$ on $M$, which corresponds to the dual coaction on $W^*_\red(\A'')$, where $\A$ is the spectral Fell bundle of $\delta$. This perspective aligns with the analogous construction for group actions described in \cite{BssEff_amenability}, and is compatible with the isomorphism $\bk_\wstar(\A'') \cong \bk(\A)''$ established in \cite[Theorem~5.23]{abadie2021amenability}.
\end{remark}

\section{Reductions and quotients of Fell bundles}\label{sec:applications}

In this section we prove the analogue of Theorem~\ref{thm: M amena iff N and Mh amenable} 
for Fell bundles over discrete groups. We begin by constructing the objects involved in the statement.

Let $\cB=\{B_t\}_{t\in G}$ be a Fell bundle over a discrete group $G$, and let 
$H\trianglelefteq G$ be a normal subgroup. 
The natural inclusion $\iota_\cB\colon \cB\to C^*(\cB)$ sends $b\in B_t$ to the cross-section 
that takes the value $b$ at $t$ and $0$ elsewhere. 
For each $u\in G/H=:\dot G$, let $\dot{B}_u$ be the norm-closure of 
$\sum_{t\in u}\iota_\cB(B_t)$ in $C^*(\cB)$. 
Note that $\iota_\cB$ induces a grading of $C^*(\cB)$ over $G$, which we call the
\emph{canonical $G$-grading of $C^*(\cB)$} induced by~$\cB$.

\begin{proposition}[cf.~Proposition~\ref{prop: the partial cross sectional w bundle}]
\label{prop: construction of dot B}
There exists a unique injective $*$-homomorphism 
$h\colon C^*(\cB_H)\to C^*(\cB)$ such that 
$h\circ \iota_{\cB_H}=\iota_\cB|_{\cB_H}$. 
Hence $h(C^*(\cB_H))$ is $C^*$-isomorphic to $\dot{B}_H$.
With the norm and $*$-operations inherited from $C^*(\cB)$, 
$\dot\cB := \{\dot{B}_u\}_{u\in \dot{G}}$ is a Fell bundle. 
Moreover, $\cB$ and $\iota_{\dot{\cB}}$ induce a grading of $C^*(\dot{\cB})$ over $G$ 
with fibres $\{B_t\}_{t\in G}$: for each $t\in G$, 
$B_t$ is linearly and isometrically identified with 
$\iota_{\dot{\cB}}(\iota_\cB(B_t))\subseteq C^*(\dot{\cB})$. 
There exists a $*$-isomorphism $k\colon C^*(\cB)\to C^*(\dot{\cB})$ 
that restricts to the identity on each fibre $B_t$.
\end{proposition}

\begin{proof}
Every $*$-representation of $\cB$ on a Hilbert space restricts to a 
$*$-representation of $\cB_H$.
Passing to integrated forms gives a $*$-homomorphism 
$h\colon C^*(\cB_H)\to C^*(\cB)$ satisfying 
$h\circ \iota_{\cB_H}=\iota_\cB|_{\cB_H}$.
To show $h$ is faithful, let $T$ be a nondegenerate $*$-representation of $\cB_H$ 
with faithful integrated form $\tilde{T}\colon C^*(\cB_H)\to \bB(X)$. 
By \cite[Theorem~2.1]{Ferraro_2024}, $T$ is $\cB$-positive and can be induced to a 
$*$-representation $S\colon \cB\to \bB(Y)$ by the process in 
\cite[Ch.~XI]{FlDr88}. 
Since $\dot{G}$ is discrete, \cite[XI~14.21]{FlDr88} implies that $T$ 
is a subrepresentation of $S|_{\cB_H}$. 
Hence $\tilde{T}\circ h=\tilde{S}$ is faithful, and $h$ is injective.

The fact that $\dot{\cB}$ is a Fell bundle follows immediately from 
$(\dot{B}_u)^*=\dot{B}_{u^{-1}}$ and $\dot{B}_u\dot{B}_v\subseteq \dot{B}_{uv}$ 
for all $u,v\in\dot{G}$.

Given any $*$-representation $T\colon C^*(\dot{\cB})\to \bB(X)$, 
there exists a unique $*$-representation $T'\colon C^*(\cB)\to \bB(X)$ 
such that $T'\circ\iota_\cB(b)=T\circ\iota_{\dot{\cB}}\circ\iota_\cB(b)$ 
for all $b\in \cB$. 
This defines a $*$-homomorphism $k\colon C^*(\cB)\to C^*(\dot{\cB})$ 
preserving the $G$-grading, hence surjective.
If $T$ is a faithful $*$-representation of $C^*(\cB)$, then 
$T\circ \iota_{\dot{\cB}}\colon \dot{\cB}\to \bB(X)$ integrates to 
a representation $\dot{T}$ satisfying $\dot{T}\circ k=T$.
Therefore $k$ is injective.
\end{proof}

\begin{theorem}\label{thm: B AD amenable iff BH and dotB AD amenable}
Let $\cB$ be a Fell bundle over $G$. 
Then $\cB$ is C$^*$-amenable if and only if both $\cB_H$ and $\dot{\cB}$ are C$^*$-amenable.
\end{theorem}

\begin{proof}
The implication 
``$\cB$ is C$^*$-amenable $\Rightarrow$ $\cB_H$ and $\dot\cB$ are C$^*$-amenable’’ 
is proved (for general locally compact groups) in \cite[Proposition~4.10]{BussFerraro25discrete}.

Conversely, assume that $\cB_H$ and $\dot\cB$ are C$^*$-amenable.
By \cite[Definition~5.16]{abadie2021amenability}, $\cB$ is C$^*$-amenable 
if the bundle of biduals $\cB''=\{B_t''\}_{t\in G}=:\cM$ is W$^*$-amenable.
Note that $\cM_H=(\cB_H)''$. 
Then Theorem~\ref{thm: M amena iff N and Mh amenable} implies that 
$\cB$ is C$^*$-amenable if and only if $\cB_H$ is C$^*$-amenable and 
$\dot{\cM}$ is W$^*$-amenable.

Let $\pi\colon C^*(\cB)\to \bB(X)$ be the universal representation 
and $\pi''\colon C^*(\cB)''\to \bB(X)$ its normal extension. 
Since $C^*(\cB)=C^*(\dot{\cB})$, we may view $(\dot{B}_u)''$ as the $\wstar$-closure of 
$\sum_{t\in u}B_t$ in $C^*(\cB)''$, and $\pi''$ restricts to a representation 
$\dot{T}\colon (\dot{\cB})''\to \bB(X)$. 
To construct $\dot{\cM}$, we need a faithful representation of $W^*_\red(\cM)$.
The bidual $B_t''$ is the $\wot$-closure of $B_t$ in $C^*(\cB)''$, yielding a 
$*$-representation $T\colon \cM\to \bB(X)$ defined by $T_m=\pi''(m)$, 
faithful on $B_e''$.
Using the faithful representation 
$\lambda\otimes T\colon W^*_\red(\cM)\to \bB(\ell^2(G)\otimes X)$, 
we see that for each $u\in\dot{G}$, 
$W^*_\red(\cM_u)$ is isomorphic to the $\wot$-closure of 
$\sum_{t\in u}\lambda_t\otimes T(B_t)$.
Moreover, $\lambda\otimes T|_{\cB}\colon \cB\to \bB(\ell^2(G)\otimes X)$ integrates to 
$\lambda\pi\colon C^*(\cB)\to \bB(\ell^2(G)\otimes X)$, whose normal extension 
$(\lambda\pi)''$ satisfies $(\lambda\pi)''(b)=\lambda_t\otimes T_b$ for 
$b\in B_t$. 
Hence $(\lambda\pi)''((\dot{B}_u)'')=W^*_\red(\cM_u)$, providing a 
multiplicative, involutive, fibrewise normal, contractive and surjective map 
$\mu\colon (\dot{\cB})''\to \dot{\cM}$.

Recall that for W$^*$-Fell bundles, W$^*$-amenability is equivalent to the W$^*$AP.
Given a net of finitely supported functions 
$\{\xi_i\colon \dot{G}\to (\dot{B}_H)''\}_{i\in I}$ witnessing the W$^*$AP of 
$(\dot{\cB})''$, the net $\{\mu\circ \xi_i\}_{i\in I}$ witnesses the W$^*$AP of $\dot{\cM}$.
\end{proof}

\begin{remark}
The implication ``$\cB$ C$^*$-amenable $\Rightarrow$ $\cB_H$ C$^*$-amenable’’ 
answers a question raised by Echterhoff and Quigg 
\cite[comments before Theorem~6.5]{echterhoff1999induced}.     
\end{remark}

The implication 
``$\cB$ is C$^*$-amenable $\Rightarrow$ $\cB_H$ and $\dot\cB$ are C$^*$-amenable’’ 
was proved in \cite{BussFerraro25discrete} using a characterization of the approximation property 
due to Bédos and Conti (BCAP) \cite{bedos2024positiveMZ}. 
For completeness, we recall this connection here, since it has been used implicitly above.

\begin{theorem}[Section~3 of \cite{BussFerraro25discrete}]
For a Fell bundle over a locally compact group, 
the Bédos–Conti approximation property (BCAP) is equivalent to 
the Exel–Ng approximation property.
\end{theorem}

After this theorem and \cite[Theorem~6.10]{abadie2021amenability}, it follows that for Fell bundles over discrete groups, the BCAP characterizes C$^*$-amenability.
An advantage of the BCAP in the discrete setting is that it is formulated in terms of 
maps $\Phi\colon \cB\to \cB$ that are restrictions of completely positive contractions 
$\tilde{\Phi}\colon C^*(\cB)\to C^*(\cB)$ preserving the $G$-grading induced by 
$\iota_\cB\colon \cB\to C^*(\cB)$ \cite[Theorem~3.11]{bedos2024positiveMZ}.
Such a map $\tilde{\Phi}$ also preserves the induced $\dot{G}$-grading of 
$C^*(\dot{\cB})$ via $\iota_{\dot{\cB}}$, 
and restricting $\tilde{\Phi}$ to $\dot{\Phi}\colon \dot{\cB}\to \dot{\cB}$ 
transports the BCAP from $\cB$ to $\dot{\cB}$.

\section{Amenability for Green twisted actions}

Let $(B,G,N,\alpha,\tau)$ be a Green twisted action with $G$ a discrete group, that is, 
$\alpha\colon G\to\Aut(B)$ is an action by C$^*$-automorphisms and 
$\tau\colon N\to U\M(B)$ is a homomorphism satisfying
\[
\alpha_s(\tau_n)=\tau_{sns^{-1}}
\quad\text{and}\quad 
\alpha_n=\Ad(\tau_n)\qquad(s\in G,\ n\in N).
\]
See \cite[Definition~5.3]{echterhoff1999induced} for details.

Consider the usual semidirect Fell bundle $B\times G\to G$ with operations 
$(b,s)(c,t)=(b\,\alpha_s(c),st)$ and $(b,s)^*=(\alpha_{s^{-1}}(b^*),s^{-1})$. 
There is a right action of $N$ given by
\[
(b,s)\cdot n := (b\,\tau_n,\ n^{-1}s),\qquad n\in N.
\]
The \emph{twisted semidirect Fell bundle} is then the orbit bundle
\[
B\times_N G := (B\times G)/N \ \longrightarrow\ G/N,\qquad [b,s]\longmapsto sN.
\]
In particular, the fibre over $sN$ is 
$(B\times_N G)_{sN}=\{[b,s]:b\in B\}$ 
\cite[Definition~5.3]{echterhoff1999induced}. 
The twisted crossed products (full and reduced) coincide with the full and reduced 
cross-sectional C$^*$-algebras of $B\times_N G$. 
Moreover, the semidirect bundle $B\times G\to G$ is isomorphic to the 
\emph{pull-back} $q^*(B\times_N G)$ via the quotient map $q\colon G\to G/N$, 
explicitly,
\[
(b,s)\longmapsto ([b,s],s)
\]
is an isomorphism of Fell bundles \cite[Theorem~5.6]{echterhoff1999induced}.

Each fibre of $B\times_N G$ is canonically isomorphic to $B$. Indeed,  
for $sN\in G/N$ the map
\[
B\longrightarrow (B\times_N G)_{sN},\qquad b\mapsto [b,s],
\]
is a bijection. Injectivity follows since $[b,s]=[b',s]$ implies 
$(b',s)=(b\tau_n,n^{-1}s)$ for some $n\in N$, hence $n=e$ and $b'=b$; 
surjectivity follows by writing any representative $(c,t)$ with $t=n^{-1}s$ 
and observing that $(c,t)=(c\tau_{n^{-1}},s)\cdot n$.

\begin{definition}
We say that the Green twisted action $(B,G,N,\alpha,\tau)$ is 
\emph{amenable} if the associated Fell bundle 
$\mathcal{D}:=B\times_N G\to G/N$ has 
the approximation property (AP). 
\end{definition}

\begin{remark}
Via the identification $B\cong (\mathcal{D})_N$, $b\leftrightarrow [b,e]$, 
the AP translates into the existence of a bounded net 
$\{a_i\colon G/N\to B\}_{i\in I}\subset L^2(G/N,B)$ of finitely supported functions 
such that for all $(s,b)\in G\times B$,
\[
\lim_i \big\|\,b-\sum_{u\in G/N} a_i(sN u)^*\,b\,\alpha_s(a_i(u))\big\| = 0.
\]
\end{remark}

If a Green twisted action $(B,G,N,\alpha,\tau)$ is amenable, then it satisfies the weak containment property, and hence
\[
B \rtimes_{\alpha,\tau} G \;\cong\; B \rtimes_{\alpha,\tau,r} G.
\]

\begin{proposition}\label{prop: twisted green action amenable and action on the center of bidual}
Let $(B,G,N,\alpha,\tau)$ be a Green twisted action and let 
$\mathcal{D}=B\times_N G\to G/N$ be the associated twisted semidirect Fell bundle. 
Then $\mathcal{D}$ has the approximation property if and only if the induced action
\begin{equation}\label{equ: quotient action}
G/N \curvearrowright Z(B''), \qquad sN\cdot z = \alpha_s''(z),    
\end{equation}
is W$^*$-amenable.
\end{proposition}

\begin{proof}
By \cite[Thm.~6.10]{abadie2021amenability}, a Fell bundle has the AP if and only if 
the associated partial action on the centre of the bidual of its unit fibre is W$^*$-amenable. 
Since $\mathcal{D}$ is saturated and every fibre is isomorphic to $B$, the associated partial action 
reduces to a global action $\theta$ of $G/N$ on $Z(B'')$.

More precisely, fix $s\in G$ and view the fibre $\mathcal{D}_{sN}$ as a $B$--$B$ imprimitivity bimodule. 
The right $B$-action is the usual one, while the left action is given by $\alpha_s$. 
Passing to biduals, $\mathcal{D}_{sN}''$ becomes a $B''$--$B''$ equivalence bimodule with left action 
$b\cdot x=\alpha_s''(b)\,x$. 
By \cite[Remark~4.5]{abadie2021amenability}, such a bimodule determines an isomorphism 
$\pi_s\colon Z(B'')\to Z(B'')$ satisfying $xa=\pi_s(a)x$ for all $x\in \mathcal{D}_{sN}''$ and $a\in Z(B'')$. 
Since the left action is $\alpha_s''$, one checks that $\pi_s=\alpha_{s^{-1}}''|_{Z(B'')}$. 
Hence the automorphism implemented by the fibre $sN$ is 
$\theta_{sN}=(\pi_s)^{-1}=\alpha_s''|_{Z(B'')}$, giving the global action in 
\eqref{equ: quotient action}.
\end{proof}

The proof above illustrates the type of reasoning that emerges from connecting the theories of
amenability in the settings of C$^*$-actions, W$^*$-actions, Fell bundles and W$^*$-Fell bundles.
Here, Anantharaman-Delaroche's classical notion of amenability
\cite{ADaction1979,ADactionII1982,ADsystemes1987}
is extended to Green twisted actions through the unifying framework of Fell bundles and their W$^*$ counterparts.

These connections are not accidental.
In \cite{abadie2021amenability}, the definition of W$^*$-amenable W$^*$-Fell bundle
was introduced in terms of the amenability of the partial or global W$^*$-actions naturally associated to each bundle, 
and shown to be equivalent to the W$^*$-approximation property (W$^*$AP). 
A Fell bundle is said to be C$^*$-amenable if its bidual W*-Fell bundle is W$^*$-amenable; and the AP characterizes C$^*$-amenability 
\cite[Theorem~6.10]{abadie2021amenability}.
With these definitions, a W$^*$ (resp.\ C$^*$) action is amenable in the sense of 
Anantharaman-Delaroche if and only if its semidirect product bundle is W$^*$-
(resp.\ C$^*$-)amenable.

\bibliographystyle{plain}

\end{document}